\documentclass[11pt]{article}
\usepackage[textwidth=150mm,textheight=220mm,centering]{geometry}
\usepackage{tikz}
\usepackage{xcolor}
\usepackage{mathrsfs}
\usepackage{amsfonts,amssymb}
\usepackage{amsmath}
\usepackage{cite}
\numberwithin{equation}{section}
\allowdisplaybreaks[4]
\newtheorem{Theorem}{Theorem}[section]
\newtheorem{Proposition}[Theorem]{Proposition}
\newtheorem{Corollary}[Theorem]{Corollary}
\newtheorem{Lemma}[Theorem]{Lemma}
\newtheorem{Example}[Theorem]{Example}

\def\bbZ{\mathbb{Z}}
\def\bbR{\mathbb{R}}

\begin{document}
\title{Hardy-Littlewood-Sobolev Inequality on Mixed-Norm Lebesgue Spaces\thanks{This work was partially supported by the
National Natural Science Foundation of China (11525104, 11531013, 11761131002, 11801282 and 11926342).}}

\author{Ting Chen\quad and \quad Wenchang Sun\\
\small School of Mathematical Sciences and LPMC,  Nankai University, Tianjin~300071,   China \\
Emails:\,
t.chen@nankai.edu.cn;\quad sunwch@nankai.edu.cn}

\date{}
\maketitle

\textbf{Abstract}.\,
We consider the Hardy-Littlewood-Sobolev inequality on mixed-norm Lebesgue spaces.
We give a complete characterization of
indices $\vec p$ and $\vec q$ such that the Riesz potential is bounded from $L^{\vec p}$ to $L^{\vec q}$,
including all the endpoint cases.
As a result, we get the mixed-norm
Hardy-Littlewood-Sobolev inequality.

\textbf{Keywords}.\,\,
Riesz potentials; Fractional integrals; mixed norms; Hardy-Littlewood-Sobolev inequality.

\section{Introduction and The Main Results}
Let $n$ be a positive integer and  $0<\lambda<n$. The  Riesz potential
of a measurable function $f$ on $\bbR^n$  is defined by
\[
   \int_{\bbR^n} \frac{f(y)}{|x-y|^{\lambda}} dy.
\]
The boundedness of the Riesz potential  can be found in many textbooks, e.g.,
see
\cite[Theorem 1.2.3]{Grafakos2014m},
\cite[Proposition 7.8]{Muscalu2013i}
or
\cite[Chapter 5.1]{Stein1970}.

\begin{Proposition}\label{prop:fractional}
Let   $0<\lambda<n$
and
$1\le p,q \le \infty$.
Then  the Riesz potential is bounded from
  $L^p(\bbR^n)$  to $L^q(\bbR^n)$ if and only if $1<p<q<\infty$ and %$1/p = 1/q + (n-\lambda)/n$.
$\lambda=n/p'+n/q$.
\end{Proposition}

An equivalent statement of the boundedness of the Riesz potential
is the Hardy-Littlewood-Sobolev inequality. Specifically,
for any $f\in L^p$ and $g\in L^q$ with $1<p,q<\infty$ and $1/p+1/q>1$, we have
\[
   \left|\int_{\bbR^n} \frac{f(y)g(x)dxdy}
   {|x-y|^{n/p'+n/q' }}\right|
   \lesssim \|f\|_{L^p} \|g\|_{L^q}.
\]

In this paper, we focus on the Hardy-Littlewood-Sobolev inequality on
mixed-norm Lebesgue spaces.

Given a multi-index  $\vec p = (p_1,\ldots,p_m)$
with $0<p_i\le \infty$,
the mixed-norm Lebesgue space  $L^{\vec p}(\bbR^{n_1}\times \ldots \times
\bbR^{n_m}\!)$ consists of all measurable functions $f\!(x_1,\ldots,x_m)$ defined on $\bbR^{n_1}\times \ldots \times
\bbR^{n_m}$ for which
\[
  \|f\|_{L^{\vec p}} :=   \Big\|  \| f\|_{L_{x_1}^{p_1}} \cdots \Big\|_{L_{x_m}^{p_m}}<\infty.
\]
For convenience,  we also write the $L^{\vec p}$
norm as
$\|\cdot\|_{L_{x_m}^{p_m}(\ldots (L_{x_1}^{p_1}))}$ or $\|\cdot\|_{L_{(x_1,\ldots,x_m)
}^{(p_1,\ldots,p_m)}}$.

Mixed-norm Lebesgue spaces are widely used in the study of
partial differential equations.
Many works have been
done since Benedek and Panzone~\cite{Benedek1961} introduced
the mixed-norm Lebesgue spaces.
It was shown that mixed-norm Lebesgue spaces
share similar properties with the ordinary Lebesgue spaces. Many classical operators are bounded on these spaces. We refer the readers to
\cite{Benedek1962,
Fernandez1987,
Hart2018,
Hormander1960,
Kurtz2007,
Rubio1986,
Stefanov2004,
Torres2015} for details.
Other generalizations on mixed-norm spaces, which include
Hardy spaces, Trieble-Lizorkin spaces, Besov spaces and Sobolev spaces can be
found in \cite
{
BandaliyevSerbetci2018,
Boggarapu2017,
CarneiroOliveiraSousa2019,
ChenSun2017,
Ciaurri2017,
Cleanthous2017,
Cordoba2017,
Georgiadis2017,
Ho2018,
HuangLiuYangYuan2018,
HuangLiuYangYuan2019,
Johnsen2015,
KarapetyantsSamko2018,
Lechner2018,
LiStinga2017,
Sandik2018,
WeiYan2018}.
See also a survey paper \cite{HuangYang2019} on aspects of mixed-norm spaces.

Given $0<\lambda<N_m$ and $f\in L^{\vec p}$,
where
\[
    N_m = n_1+\ldots +n_m,
\]
 the Riesz potential of $f$ is defined by
\[
  I_{\lambda}f(x_1,\ldots,x_m) =
           \int_{\bbR^{N_m}} \frac{f(y_1,\ldots,y_m)}{(\sum_{i=1}^m |x_i-y_i|)^{\lambda}} dy_1\ldots dy_m.
\]
A natural problem arises: find all indices $\vec p$ and $\vec q$
such that $I_{\lambda}$ is bounded from $L^{\vec p}$ to $L^{\vec q}$.

In 1961,
  Bededek and   Panzone    \cite{Benedek1961} showed that for
$\vec p=(p_1,\ldots,p_m)$ and $\vec q=(q_1,\ldots,q_m)$
satisfying $1<p_i<q_i<\infty$, $1\le i\le m$,
$I_{\lambda}$ is bounded from $L^{\vec p}$ to $L^{\vec q}$, where $\lambda = \sum_{i=1}^m (n_i/q_i+n_i/p'_i)$.

In 1974, Adams and Bagby \cite{AdamsBagby1974} showed that
if
\[
  \begin{cases}
  1\le p_i\le q_i\le \infty, & 1\le i\le m-1,\\
   1<p_m<q_m<\infty
  \end{cases}
\]
or
\[
  \begin{cases}
   1\le p_i\le q_i\le \infty, & 1\le i\le v,\\
   1<p_v<q_v<\infty, & \\
   1<p_i=q_i<\infty, &v+1\le i\le m
  \end{cases}
\]
for some $1\le v<m$,
then $I_{\lambda}$ is bounded from $L^{\vec p}$ to $L^{\vec q}$.

For the necessary conditions, they showed that whenever
$1<p_i<\infty$ and $q_i\ge 1$ for all $1\le i\le m$,
it is necessary that $p_i\le q_i$.
Moreover, for the case $m=2$, they also illustrated that $I_{\lambda}$ is not bounded if $\vec p$ and $\vec q$ meet one of the followings,

(i)\,\,  $q_1=\infty$ and $p_2=q_2$.

(ii)\,\, $p_1\le q_1 <p_2\le q_2=\infty$.

In this paper, we give a complete characterization of indices
$\vec p$ and $\vec q$ such that $I_{\lambda}$ is bounded from $L^{\vec p}$ to $L^{\vec q}$.

Before stating the result, we introduce some notations.

Define $\Gamma_{\lambda,m}$ recursively as follows.
$\Gamma_{\lambda,1}$ consists of all vectors
$(p_1,q_1)$ for which $1<p_1<q_1<\infty$ and $1/p_1=1/q_1 + (n_1-\lambda)/n_1$.

For $m\ge 2$,  $\Gamma_{\lambda,m}$ consists of all
vectors $(\vec p, \vec q)$ which satisfy the homogeneity condition
 \begin{equation}\label{eq:homo:m}
  \frac{n_1}{p_1}+\ldots + \frac{n_m}{p_m}
  =
  \frac{n_1}{q_1}+\ldots + \frac{n_m}{q_m}
  +N_m-\lambda,
\end{equation}
 $1\le p_i\le q_i\le \infty$ for all $1\le i\le m$,
and  one of the following conditions,
\begin{enumerate}
\item[(T1)]
$1<p_m<q_m<\infty$,

\item[(T2)]
 $ p_m=1$, $q_m<\infty$ and there is some $1\le i_1\le m-1$ such that
  $1< p_{i_1}<q_{i_1}$,
  $p_i=1$ or  $p_i=q_i $   for all
   $i_1+1\le i\le m-1$, and
   $q_m\ge p_i$ for all
   $i_1 \le i\le m $,

\item[(T3)]
 $p_m>1$, $q_m=\infty$
  and there is some $1\le i_2\le m-1$ such that
    $p_{i_2}<q_{i_2}<\infty$,
    $q_i=\infty$ or $q_i=p_i $ for all $i_2+1\le i\le m-1$,
    and $p_m\le q_i$ for all $i_2 \le i\le m $,

\item[(T4)]
  $p_m=1$, $q_m=\infty$ and $(p_1,\ldots,p_{m-1},q_1,\ldots,q_{m-1})\in \Gamma_{\lambda,m-1}$,

\item[(T5)]
  $1<p_m=q_m<\infty$ and $(p_1,\ldots,p_{m-1},q_1,\ldots,q_{m-1})\in \Gamma_{\lambda-n_m,m-1}$.

\end{enumerate}

Our main result is the following.

\begin{Theorem}\label{thm:mixed fractional}
Suppose that $0<\lambda<N_m$, $\vec p=(p_1,\ldots, p_m)$
and $\vec q = (q_1, \ldots, q_m)$
with $0<p_i, q_i\le \infty$, $1\le i\le m$.
Then  $I_{\lambda}$ is bounded from $L^{\vec p}$
to $L^{\vec q}$ if and only if $(\vec p, \vec q)\in \Gamma_{\lambda,m}$.
\end{Theorem}

As a consequence,
we get the Hardy-Littlewood-Sobolev inequality for
mixed-norm Lebesgue spaces.

\begin{Theorem}\label{thm:HLS}
Let $\vec p=(p_1,\ldots, p_m)$
and $\vec q = (q_1, \ldots, q_m)$,
where $0<p_i, q_i  \le \infty$, $1\le i\le m$.
Then the  Hardy-Littlewood-Sobolev inequality
\[
 \Big | \int_{\bbR^{2N_m}}\!\!\!\! \frac{f(y_1,\ldots,y_m)g(x_1,\ldots,x_m)
     dy_1\ldots dy_m dx_1\ldots dx_m}
     {(\sum_{i=1}^m|x_i-y_i|)^{\sum_{i=1}^m (n_i/p'_i+n_i/q'_i)}}\Big|
 \lesssim \|f\|_{L^{\vec p}} \|g\|_{L^{\vec q}}
\]
is true
for any $f\in L^{\vec p}$ and $g\in L^{\vec q}$
if and only if $(\vec p, \vec q)\in \Omega_m$,
where $\Omega_m$ is defined recursively as follows.

$\Omega_1$ consists of all vectors
$(p_1,q_1)$ for which $1<p_1, q_1<\infty$ and $1/p_1+ 1/q_1 >1$.

For $m\ge 2$,  $\Omega_m$ consists of all
vectors $(\vec p, \vec q)$ which satisfy
 $1\le p_i, q_i\le \infty$,
 $1/p_i+1/q_i\ge 1$
 for all $1\le i\le m$,
and  one of the following conditions,
\begin{enumerate}
\item
$1<p_m, q_m<\infty$ and
$1/p_m + 1/q_m>1$,

\item
 $ p_m=1$, $q_m>1$, and there is some $1\le i_1\le m-1$ such that
  $ p_{i_1}>1$, $1/p_{i_1}+1/q_{i_1}>1$,
$p_i=1$ or $1/p_i +1/q_i=1$
for all $i_1+1\le i\le m-1$,
and $1/p_i + 1/q_m \ge 1$ for all $i_1 \le i\le m $,

\item
$p_m>1$, $q_m=1$,
  and there is some $1\le i_2\le m-1$ such that
    $q_{i_2}>1$, $1/p_{i_2}+1/q_{i_2}>1$,
    $q_i=1$ or $ 1/p_i+1/q_i=1$ for all $i_2+1\le i\le m-1$,
    and $1/p_m + 1/q_i\ge 1$ for all $i_2\le i\le m$,

\item
  $p_m=q_m=1$  and  $(p_1,\ldots,p_{m-1},q_1,\ldots,q_{m-1})\in \Omega_{m-1}$,

\item
$1< p_m, q_m<\infty$, $1/p_m+1/q_m=1$ and $(p_1,\ldots,p_{m-1},q_1,\ldots,q_{m-1})\in \Omega_{m-1}$.

\end{enumerate}
\end{Theorem}

To prove Theorem~\ref{thm:mixed fractional},
the main ingredients include
the
Calder\'on-Zygmund theory for
  mixed-norm Lebesgue spaces
and
 the theory of interpolation spaces,
where the reiteration theorem and various formulas  for  interpolation spaces play
an important role.

The paper is organized as follows.
In Section 2, we introduce the Fefferman-Stein inequality
for mixed-norm Lebesgue spaces.
In Section 3, we introduce the Bochner space
and give the difference between Bochner spaces
and mixed-norm Lebesgue spaces.
In Section 4, we study the interpolation between  mixed-norm Lebesgue
spaces.
We prove that for a two-variable measurable function, it is measurable with respect to one variable whenever the
norm of an interpolation space is applied to the other variable,
based on which we show that
some interpolation formulas for Bochner spaces
are also true for mixed-norm Lebesgue spaces.
In Section 5, we give proofs  for the main results.

\textbf{Notations}.
In this paper, $L^{(p,q)}$ denotes the mixed-norm Lebesgue space
$L^q(L^p)$, while $L^{p,q}$ denotes the  Lorentz space.

If there is a constant $C$ such that $A\le CB$, then we write
$A\lesssim B$. If $A\lesssim B$ and $B\lesssim A$, then we write
$A\approx B$.

\section{The Fefferman-Stein Inequality}
In this section, we introduce the Fefferman-Stein
inequality for the partial Hardy-Littlewood  maximal function.

Given a locally integral function $f$,
the Hardy-Littlewood maximal function $Mf$ is defined
by
\[
  Mf(x) = \sup_{ Q\ni x} \frac{1}{|Q|}
    \int_{Q} |f(y)| dy,
\]
where the supremum is taken over all cubes in $\bbR^n$
whose sides parallel to the axes.

It is well known that $M$ is bounded on $L^p$ when $1<p\le \infty$.
Fefferman and Stein proved a vector-valued version~\cite{FeffermanStein1971}.
Specifically,
for $1<p\le \infty$ and $1<q<\infty$,
\[
\Big\|
  \|\{M f_j:\, j\in \bbZ \}\|_{\ell_{\bbZ}^{p}}\Big\|_{L^q}
\le C_{  p, q}
\Big\|
  \|\{ f_j:\, j\in \bbZ \}\|_{\ell_{\bbZ}^{p}}\Big\|_{L^q},
\]
where $C_{p,q} = (1+1/(p-1))(q+1/(q-1))$.

Bagby~\cite{Bagby1975}  extended this inequality to a more general case.
For $1<p_i\le \infty$, $1\le i\le m$ and $1<q<\infty$,
it was shown that
\begin{equation}\label{eq:SF:e1}
\Big\|
  \|\{M f_j:\, j\in \bbZ^m\}\|_{\ell_{\bbZ^m}^{(p_1,\ldots,p_m)}}\Big\|_{L^q}
\le C_{\vec p, q}
\Big\|
  \|\{ f_j:\, j\in \bbZ^m\}\|_{\ell_{\bbZ^m}^{(p_1,\ldots,p_m)}}\Big\|_{L^q},
\end{equation}
where $C_{\vec p, q}$ is a constant depending on $\vec p$ and $q$.
Note that the above result was stated for the case
$1<p_i<\infty$ in \cite{Bagby1975}. But the proof is valid
for all $1< p_i\le \infty$, $1\le i\le m$.
See also \cite{Fernandez1987} for other generalizations of the Fefferman-Stein
inequality.

Here we present a continuous version of (\ref{eq:SF:e1}).

\begin{Proposition}\label{prop:SF}
Suppose that $1<p_i\le \infty$ and $1<q_j<\infty$, $1\le i\le m$
and $1\le j\le k$. We have
\begin{align}
  & \| M_{y_1}\ldots M_{y_k} f(x_1,\ldots,x_m,y_1,\ldots,y_k)\|_{
  L_{(x_1,\ldots,x_m,y_1,\ldots,y_k)}^{
    (p_1,\ldots,p_m,q_1,\ldots,q_k)}}\nonumber \\
  &\le C_{\vec p, \vec q}
  \| f(x_1,\ldots,x_m,y_1,\ldots,y_k)\|_{
  L_{(x_1,\ldots,x_m,y_1,\ldots,y_k)}^{
     (p_1,\ldots,p_m,q_1,\ldots,q_k)}}, \label{eq:Fefferman-Stein}
\end{align}
 where $C_{\vec p, \vec q}$ is a constant
and $M_y$ means the maximal function with respect to the
variable $y$, i.e.,
\[
  M_y f(x,y) = (M f(x, \cdot))(y).
\]
\end{Proposition}

\begin{proof}
First, we consider the case $k=1$. In this case,
(\ref{eq:Fefferman-Stein}) becomes
\begin{align}
  \| M_y f(x_1,\ldots,x_m,y)\|_{L_{(x_1,\ldots,x_m,y)}^{(p_1,\ldots,p_m,q)}}
  \le C_{\vec p, q}
  \|  f(x_1,\ldots,x_m,y)\|_{L_{(x_1,\ldots,x_m,y)}^{(p_1,\ldots,p_m,q)}}.
  \label{eq:SF:e2}
\end{align}
We prove  (\ref{eq:SF:e2}) by induction on $m$.
Note that it is true for $m=0$.
Assume that it is true for some $m\ge 0$.
Let us consider the case $m+1$.

Fix some $1<p_2,\ldots,p_m\le \infty$ and $1<q<\infty$.
For $p_1=\infty$, we see from the inductive assumption that
\begin{align}
   \| M_y f(x_1,\ldots,x_{m+1},y)\|_{L_{(x_1,\ldots,x_{m+1},y)
  }^{(p_1,\ldots,p_{m+1},
      q)}}\nonumber %\\
&\le
    \left \| M_y \|f(x_1,\ldots,x_{m+1},y)\|_{L_{x_1}^{\infty}}
       \right \|_{L_{(x_2,\ldots,x_{m+1},y)}^{(p_2,\ldots,p_{m+1},
      q)}} \nonumber \\
 & \lesssim
  \|  f(x_1,\ldots,x_{m+1},y)\|_{L_{(x_1,\ldots,x_{m+1},y)}^{
   (p_1,\ldots,p_{m+1},q)}}.
  \label{eq:SF:e3}
\end{align}

On the other hand, for any $1<p_1< \min\{p_2,\ldots,p_{m+1},q\}$,
denote
$\vec u = (p_2/p_1$, $\ldots$, $p_{m+1}/p_1$,
      $q/p_1)$.
%When $m=0$, this means that $1<p_1<q$ and $\vec u = (q/p_1)$.
Note that for any $g\in L^{p_1}$
and non-negative measurable functions $w$, we have
\begin{equation}\label{eq:ea:a10}
  \Big( \int_{\bbR^n} |Mg(x)|^{p_1} w(x) dx  \Big )^{1/p_1}
  \le C_{p_1}
  \Big( \int_{\bbR^n} |g(x)|^{p_1} Mw(x) dx  \Big )^{1/p_1},
\end{equation}
where $C_{p_1}$ is a constant (see \cite[Theorem 2.16]{Duoandikoetxea2001}).

For any $f\in L_{(x_1,\ldots,x_{m+1},y)}^{(p_1,\ldots,p_{m+1},
      q)}$, we have
\begin{align}
 & \| M_y f(x_1,\ldots,x_{m+1},y)\|^{p_1}_{
    L_{(x_1,\ldots,x_{m+1},y)}^{(p_1,\ldots,p_{m+1},
      q)}}\nonumber \\
&=
    \left \| \|M_y f(x_1,\ldots,x_{m+1},y)\|^{p_1}_{L_{x_1}^{p_1}}
        \right\|_{L_{(x_2,\ldots,x_{m+1},y)}^{\vec u}} \nonumber \\
&= \! \sup_{\|h\|_{L^{\vec u\ \!\!'}}=1}
    \int \!\! | M_yf(x_1, \ldots, x_{m+1},y)|^{p_1}
        h(x_2, \ldots, x_{m+1},y) dx_1\!\ldots\! dx_{m+1} dy .
        \label{eq:ea:a11}
\end{align}
Applying Fubini's theorem and (\ref{eq:ea:a10}), we get
\begin{align}
&      \int | M_yf(x_1,\ldots,x_{m+1},y)|^{p_1}
        |h(x_2,\ldots,x_{m+1},y)| dx_1\ldots dx_{m+1} dy \nonumber \\
 &\lesssim
        \int | f(x_1,\ldots,x_{m+1},y)|^{p_1}
       M_y |h(x_2,\ldots,x_{m+1},y)| dx_1\ldots dx_{m+1} dy
       \nonumber \\
 &=   \int \| f(x_1,\ldots,x_{m+1},y)\|^{p_1}_{L_{x_1}^{p_1}}     M_y |h(x_2,\ldots,x_{m+1},y)| dx_2\ldots dx_{m+1} dy\nonumber
    \\
 &\le \left\|
\| f(x_1,\ldots,x_{m+1},y)\|^{p_1}_{L_{x_1}^{p_1}}
  \right\|_{L_{(x_2,\ldots,x_{m+1},y)}^{\vec u}}
       \|M_yh(x_2,\ldots,x_{m+1},y)\|_{L^{\vec u\ \!\!'}} ,
        \label{eq:ea:a17}
\end{align}
where we use H\"older's inequality in the last step.
By the inductive assumption,
\[
 \|M_yh(x_2,\ldots,x_{m+1},y)\|_{L^{\vec u\ \!\!'}}
 \lesssim
 \| h \|_{L^{\vec u\ \!\!'}}.
\]
It follows from (\ref{eq:ea:a11}) and (\ref{eq:ea:a17}) that
\begin{align}
 & \| M_y f(x_1,\ldots,x_{m+1},y)\|_{L_{(x_1,\ldots,x_{m+1},y)}^{(p_1,\ldots,p_{m+1},
      q)}}\nonumber \\
&\lesssim  \sup_{\|h\|_{L^{\vec u\ \!\!'}}=1}
\left\|
\| f(x_1,\ldots,x_{m+1},y)\|^{p_1}_{L_{x_1}^{p_1}}
  \right\|^{1/p_1}_{L_{(x_2,\ldots,x_{m+1},y)}^{\vec u}}
       \| h \|_{L^{\vec u\ \!\!'}}^{1/p_1}
        \nonumber \\
&= \|f\|_{L^{(p_1,\ldots,p_{m+1},q)}}.\nonumber
\end{align}
By the interpolation theorem for mixed-norm Lebesgue spaces
\cite[Theorem 7.2]{Benedek1961},
we get that   (\ref{eq:SF:e3}) is true for all
$1< p_1,\ldots,p_{m+1}\le \infty$ and $1<q<\infty$.
By induction, (\ref{eq:SF:e2}) is true for all $m\ge 0$.

Now assume that (\ref{eq:Fefferman-Stein}) is true for some $k\ge 1$.
For the case $k+1$, we see from the inductive assumption
that for any $y_{k+1}$,
\begin{align*}
  & \| M_{y_1}\ldots M_{y_k} M_{y_{k+1}} f(x_1,\ldots,x_m,y_1,\ldots,y_k,y_{k+1})\|_{
  L_{(x_1,\ldots,x_m,y_1,\ldots,y_k )}^{
    (p_1,\ldots,p_m,q_1,\ldots,q_k )}}  \nonumber
    \\
 &
  \lesssim
 \|  M_{y_{k+1}} f(x_1,\ldots,x_m,y_1,\ldots,y_k,y_{k+1})\|_{
  L_{(x_1,\ldots,x_m,y_1,\ldots,y_k  )}^{
    (p_1,\ldots,p_m,q_1,\ldots,q_k )}}.
\end{align*}
It follows from (\ref{eq:SF:e2}) that
\begin{align*}
  & \| M_{y_1}\ldots M_{y_k} M_{y_{k+1}} f(x_1,\ldots,x_m,y_1,\ldots,y_k,y_{k+1})\|_{
  L_{(x_1,\ldots,x_m,y_1,\ldots,y_k,y_{k+1} )}^{
    (p_1,\ldots,p_m,q_1,\ldots,q_k,q_{k+1} )}}  \nonumber
    \\
 &
  \lesssim
  \Big \|
 \|  M_{y_{k+1}} f(x_1,\ldots,x_m,y_1,\ldots,y_k,y_{k+1})\|_{
  L_{(x_1,\ldots,x_m,y_1,\ldots,y_k  )}^{
    (p_1,\ldots,p_m,q_1,\ldots,q_k )}} \Big\|_{L_{y_{k+1}}^{q_{k+1}}}\\
&\lesssim
 \|   f \|_{  L^{
    (p_1,\ldots,p_m,q_1,\ldots,q_k,q_{k+1})}}.
\end{align*}
By induction,  (\ref{eq:Fefferman-Stein}) is true for any $m\ge 0$ and $k\ge 1$.
This completes the proof.
\end{proof}

\section{Bochner and Mixed-Norm Lebesgue Spaces}
Let $X$ be a quasi-normed vector space.
For any $0< p\le \infty$, we denote by $\mathcal L^p(X)$  the Bochner space
consisting of all strongly measurable functions $f$ which
are defined on $\bbR^n$ and take values in a separable subspace
of $X$ such that
$x\mapsto \|f(x)\|_X$ belongs to the ordinary $L^p$ space.

Recall that a function $f$: $\bbR^n\rightarrow X$ is strongly measurable \cite{Hytonen2016}
if there exists a sequence of functions
of the form $f_k = \sum_{i=1}^{r} a_i \chi^{}_{E_i}$,
where $a_i\in X$, $E_i$ are pairwise disjoint measurable
sets in $\bbR^n$ and are of finite measures, $1\le i\le r$ and $r$
is a positive integer, such that
$\lim_{k\rightarrow \infty} f_k(x) = f(x)$ for almost
all $x\in\bbR^n$.

To avoid confusion, we use the following convention:
$\mathcal L^p(X)$ (with a calligraphic letter $\mathcal L$ ) stands for the  Bochner spaces,
while all other compositions of function spaces,
such as $L^p(X)$ or $L^{p,q}(X)$, stand for mixed-norm spaces,
where $X$ is a quasi-normed vector space
consisting of some measurable functions defined on $\bbR^N$.

Suppose that $0<p\le \infty$
and $\vec u = (u_1,\ldots,u_m)$ with $0 <u_i\le\infty$, $1\le i\le m$.
We point out that the Bochner space $\mathcal L^p(L^{\vec u})$ is different from the mixed-norm
Lebesgue space $L^p(L^{\vec u})$.
In fact, we see from their definitions that the mixed-norm function space $L^p(L^{\vec u})$ consists of measurable functions $f(x,y)$ defined on $\bbR^{n_1+\ldots+n_m}\times \bbR^n$,
while the
Bochner space $\mathcal L^p(L^{\vec u})$
consists of functions $f(x,y)$ for which
$y\mapsto f(\cdot,y)$ is strongly measurable
from $\bbR^n$ to $L^{\vec u}$.
In other words, they
consist of functions with different measurability.
So they are different.

The following are two  examples.

\begin{Example}\label{Ex:ex1}\upshape
Set   $f = \chi^{}_{\{0\le x\le y\le 1\}}$.
It is easy to see that $f$ is measurable
and $f\in L^{p}(L^{\infty})$. However, since the range of
$y\mapsto f(\cdot,y)$ is $\{\chi^{}_{[0,y]}:\, 0\le y\le 1\}$, which is a non-separable subset of $L^{\infty}$,
we have $f\not\in \mathcal L^{p}(L^{\infty})$.
\end{Example}

\begin{Example}\upshape
The example by Sierpinski \cite[Chapter 8]{Rudin1987}, which depends on  the continuum hypothesis,
shows
that
there is a function $f$ on $[0,1]^2$
such that for every $x\in [0,1]$, $f(x, y)=1$
for almost all $y\in [0,1]$,
while for every $y\in [0,1]$, $f(x, y)=0$
for almost all $x\in [0,1]$.
Set $f_k\equiv 0$. Then for any $y\in [0,1]$,
$\| f_k(\cdot,y) - f(\cdot,y)\|_{L_{x}^{q}}=0$.
Therefore, $y\mapsto f(\cdot,y)$ is a strongly measurable
function  and equals to $0$ in $\mathcal L^{p}(L^{q})$.

On the other hand, since
\[
  \int_0^1 dy \int_0^1 f(x,y)dx =0
  \ne 1 =\int_0^1 dx \int_0^1 f(x,y)dy,  
\] 
$f$ is not measurable on $[0,1]^2$
and therefore $f\not\in L^p(L^q)$ for any $0<p,q\le \infty$.
\end{Example}

Recall that elements in $\mathcal L^p(L^{\vec u})$
are indeed equivalent classes of functions, i.e.,
$f=g$ whenever $\|f-g\|_{\mathcal L^p(L^{\vec u})}=0$.
%
%Nevertheless, we see from the above arguments that for almost all
%$y$, $f(x,y)=0$ for almost all $x$.
The above examples 
show that functions  in $\mathcal L^p(L^{\vec u})$ might be 
not measurable.
Moreover, we have   the following characterization of Bochner spaces,
which was proved in \cite{Benedek1961}
for the case $p, u_i\ge 1$. But the arguments also work for
the general case $p, u_i>0$.

\begin{Proposition}[{\cite[\S9]{Benedek1961}}]
\label{prop:characterization vv Lebesgue}
Suppose that $0<p\le \infty$
and $\vec u = (u_1,\ldots,u_m)$ with $0 < u_i\le \infty$, $1\le i\le m$. Then for any $f\in \mathcal L^p(L^{\vec u})$,
 there is  some $f_0\in L^p(L^{\vec u})$ such that
 for almost all $y\in \bbR^n$,
$
  f(x,y) = f_0(x,y)$ for almost all $x\in\bbR^{n_1+\ldots+n_m}$.
Consequently, $\|f-f_0\|_{\mathcal L^p(L^{\vec u})}=0$.

Whenever $0<u_i<\infty$ for all $1\le i\le m$,
for any $f_0\in L^p(L^{\vec u})$, there is some $
f\in \mathcal L^p(L^{\vec u})$ such that
$f=f_0$, a.e. Hence $\|f-f_0\|_{L^p(L^{\vec u})}=0$.
\end{Proposition}

%Proposition~\ref{prop:characterization vv Lebesgue}
% shows that whenever $0<u_i<\infty$ for
%all $1\le i\le m$,
%$L^p(L^{\vec u})$ consists of good (measurable) equivalent
% classes of functions in $\mathcal L^p(L^{\vec u})$.

The following result can be proved with almost the same arguments as
those in the proof of \cite[Proposition 5.5.6]{Grafakos2014}
or \cite[Lemma 1.2.19]{Hytonen2016}.

\begin{Proposition}\label{prop:density}
Suppose that $0<p< \infty$
and $\vec u = (u_1,\ldots,u_k)$ with $0< u_i<\infty$, $1\le i\le k$. Then functions of the form
\begin{equation}\label{eq:ec:e35}
\sum_{i=1}^r a_i(x) \chi^{}_{E_i}(y)
\end{equation}
are dense   in $L^p(L^{\vec u})$,
where
$a_i(x)\in L^{\vec u}$, $E_i$ are pairwise disjoint measurable
sets in $\bbR^n$ with finite measures,
and $r\ge 1$ is an integer.

Moreover, the conclusion is also true whenever $L^p$ is replaced
by $L^{\vec p}$ with $\vec p = (p_1,\ldots,p_m)$,
$0<p_i<\infty$ for all $1\le i\le m$, and
$L^{\vec u}$ is replaced by any separable quasi-normed vector space of measurable functions defined on $\bbR^N$.
%provided that the simple functions are replaced by
%countably simple functions.
\end{Proposition}

%
%The boundedness of operators on Bochner spaces
%are well studied \cite{Grafakos2014,Hytonen2016}.
%However, few results are found for mixed-norm Lebesgue
%spaces.
%
%We see from Proposition~\ref{prop:characterization vv Lebesgue}
%that whenever $\vec u = (u_1,\ldots,u_m)$
%with $0<u_i<\infty$, $\mathcal L^p(L^{\vec u})=L^p(L^{\vec u})$.

%Here we show that some bounded positive operators on Bochner
% spaces have an extension on mixed-norm Lebesgue
%spaces, which is used in the proof of
%Theorem~\ref{thm:mixed fractional}.
%

\section{Interpolation between Mixed-Norm Lebesgue Spaces}

The theory of interpolation spaces is a powerful tool
in the study of the boundedness of operators. Here we introduce some fundamental results on  real interpolation spaces  which is based on the
  $K$-functional~\cite{Lions1964}.
 And we refer to \cite{Bergh1976,Cwiel1974,Hytonen2016,Janson1988,Milman1981}
for more details on this topic.

Let $X_0$ and $X_1$ be two quasi-normed vector spaces, both of which are  continuously embedded in a Hausdorff topological
vector space.
We call $(X_0,X_1)$  a pair of interpolation couples.

Fix some interpolation couple $(X_0, X_1)$.
For any $t>0$ and $x\in X_0+X_1$,
we define the $K$-functional $K(t,x;X_0,X_1)$ by
\[
  K(t,x;X_0,X_1) = \inf\{\|x_0\|_{X_0}+t\|x_1\|_{X_1}:\, x=x_0+x_1,x_0\in X_0, x_1\in X_1\}.
\]
For $0< \theta<1$ and $0< p\le \infty$, the interpolation space $(X_0, X_1)_{\theta,p}$ is defined by
\[
  (X_0, X_1)_{\theta,p} =\{x \in  X_0+X_1:  \|x\|_{\theta,p}  <\infty\},
\]
where
\begin{equation}\label{eq:theta p}
  \|x\|_{\theta,p} = \|\{2^{-n\theta}K(2^n,x;X_0,X_1):\,n\in\bbZ\}\|_{\ell^{p}}.
\end{equation}

Note that for any $t>0$,  $x\mapsto K(t,x;X_0,X_1)$ is a quasi-norm on $X_0+X_1$.
Moreover, we see from the definition of the $K$-functional
that $ K(t,x; X_0,X_1 )$ is increasing with respect to
$t$ and for any $s,t>0$,
\begin{equation}\label{eq:ec:e31}
  K(t,x;X_0,X_1)
  \le \max\{1, \frac{t}{s}\} K(s,x;X_0,X_1 ).
\end{equation}
Hence $K(t,x;X_0,X_1) $ is continuous with respect to
$t$.

Let $(X_0, X_1)$ and
$(Y_0, Y_1)$ be interpolation couples.
We call  $T: X_0 + X_1 \rightarrow Y_0 + Y_1$  a
\emph{quasi-linear operator} if for any $x_0+x_1\in X_0 + X_1$ with $x_i\in X_i$, $i=0,1$,
there exist $y_0\in Y_0$ and $y_1\in Y_1$ such that
$T(x_0+x_1) = y_0+y_1$ and
\begin{equation}\label{eq:C}
 \|y_i\|_{Y_i} \le C_i \|x_i\|_{X_i}, \qquad i=0,1
\end{equation}
for some constants $C_0, C_1$.

It is easy to check that if $T$ is a bounded linear operator
from $X_i$ to $Y_i$, $i=0,1$,
then it is quasi-linear from $X_0+X_1$ to $Y_0+Y_1$.

For any $1< p_0, p_1\le \infty$,
the Hardy-Littlewood maximal function $M$ is a quasi-linear operator
on $L^{p_0}+L^{p_1}$.
To see this, take some $f_i\in L^{p_i}$, $i=0,1$. Let
\[
   g_i = \frac{Mf_i}{Mf_0 + Mf_1} \cdot M(f_0+f_1),\qquad i=0,1.
\]
Then $g_0+g_1 = M(f_0+f_1)$. Since $M(f_0+f_1) \le Mf_0 + Mf_1 $,
we have
$\|g_i\|_{L^{p_i}}\le \|Mf_i\|_{L^{p_i}}\lesssim \|f_i\|_{L^{p_i}}$, $i=0,1$.

The following result shows that interpolation spaces
are very useful in the study of the boundedness of operators.
A proof for linear operators acting on Banach spaces can be found in \cite[Theorem C 3.3]{Hytonen2016}.
For the case when $T$ is a quasi-linear operator and
$X_i, Y_i$ are quasi-normed vector spaces, see
Sagher~\cite[Theorem 16]{Sagher1972} or
\cite[\S3.13, Exercise 16]{Bergh1976}.

\begin{Proposition}
\label{prop:interpolation}
  Let $(X_0, X_1)$ and
$(Y_0, Y_1)$ be interpolation couples.
Suppose that $T: X_0 + X_1 \rightarrow Y_0 + Y_1$ is
 a quasi-linear
operator.
Then for all $0< \theta<1$ and $0<p\le \infty$,
the operator $T$ maps $(X_0,X_1)_{\theta,p}$
into $(Y_0,Y_1)_{\theta,p}$ and we have
\begin{equation}\label{eq:p:e5}
\|T\|_{(X_0,X_1)_{\theta,p}\rightarrow (Y_0,Y_1)_{\theta,p}} \le C_0^{1-\theta}C_1^{\theta},
\end{equation}
where $C_0$ and $C_1$ are constants defined in (\ref{eq:C}).
\end{Proposition}

To apply the above proposition, we have to characterize
the interpolation spaces $(X_0,X_1)_{\theta,p}$
and $(Y_0,Y_1)_{\theta,p}$.

The celebrated Lions-Peetre formula \cite{Lions1964}
says that for any $1\le p_0\ne p_1 \le \infty$ and
interpolation couple $(X_0, X_1)$ of Banach spaces,
we have
\[
  \big( \mathcal L^{p_0}(X_0),\mathcal  L^{p_1}(X_1)\big)_{\theta,p}
   =\mathcal  L^p \big((X_0, X_1)_{\theta,p}\big)
\]
with equivalent norms,
where $0<\theta<1$ and $1/p = (1-\theta)/p_0 + \theta/p_1$.
And Sagher \cite[Theorem 19]{Sagher1972} showed that the same conclusion is true
when $0<p_0\ne p_1\le\infty$ and $(X_0, X_1)$
is an interpolation couple of
quasi-normed vector spaces.
%And Cwikel \cite{Cwiel1974} showed that there is
%no reasonable generalization for other choices of $p$.

For the case $p_0=p_1$, we have the following result,
which was stated for Banach spaces $X_0$ and $X_1$,
but also valid for quasi-normed vector spaces $X_0$ and $X_1$.
\begin{Proposition}[\cite{Cwiel1974,Pisier1993}]
\label{prop:interpolation pq a}
Let $(X_0,X_1)$ be an interpolation couple of
complete quasi-normed vector spaces
and
$0<\theta<1$.
If $0<q\le p\le \infty$, we have the continuous embedding
\begin{equation}\label{eq:interpolation}
 (\mathcal L^p(X_0), \mathcal L^p(X_1))_{\theta,q}
 \hookrightarrow \mathcal L^p((X_0,X_1)_{\theta,q}),
\end{equation}
and the reverse  embedding
 holds if $0<p\le q\le\infty$. In particular, for
$0<p=q\le \infty$, we have
\begin{equation}\label{eq:e:p}
(\mathcal L^p(X_0), \mathcal L^p(X_1))_{\theta,p}
  =  \mathcal L^p((X_0,X_1)_{\theta,p})
\end{equation}
with equivalent norms.
\end{Proposition}

For the interpolation of Lorentz spaces, we have the following formula.
%
%which was stated for ordinary functions,
%but also valid for functions taking values
%in a quasi-normed vector space.

\begin{Proposition}[{\cite[Theorem 5.3.1]{Bergh1976}}]
\label{prop:interpolation Lorentz a}
Let   $0<p_0$, $p_1$, $q_0$, $q_1$,
 $q\le \infty$,
$0<\theta<1$ and $1/p = (1-\theta)/p_0+\theta/p_1$.
Then for $p_0\ne p_1$, we have
\[
  (L^{p_0,q_0} , L^{p_1,q_1} )_{\theta,q}
  = L^{p,q}
\]
with equivalent norms.
And the formula is also true for $p_0=p_1$
provided $1/q = (1-\theta)/q_0+\theta/q_1$.
\end{Proposition}

%\subsection{Interpolation between Mixed-Norm Spaces}
%In this subsection,
Now we consider the interpolation between mixed-norm spaces.

Let $A$ and $B$ be two quasi-normed vector spaces
which consist  of measurable functions defined on
$\bbR^m$ and $\bbR^n$, respectively.
Denote by $B(A)$ the function space which consists of all
measurable functions $f$  on $\bbR^m\times \bbR^n$ for which
$y\mapsto \|f(\cdot,y)\|_{A} $ is a function in
$B$ and
\[
\|f\|_{B(A)}:= \big\|\|f(\cdot, y)\|_A\big\|_{B}<\infty.
\]

Let $X$ be a function space consisting of some
measurable functions defined on  $ \bbR^n$.
We say that  $X$  is solid  if
for any
measurable
functions $f, g$ with  $g\in X$ and $|f(x)|\le |g(x)|$ for all $x\in \bbR^n$, we have
$f\in X$ and $\|f\|_X \le \|g\|_X$.

It is easy to see that if $X$ is solid,
then $\|f\|_X = \big\| |f|\big\|_X$ for any $f\in X$.
Moreover, mixed-norm Lebesgue spaces are solid
while the space of continuous functions is not.

%
%The following proposition gives an interpolation formula
%for mixed-norm spaces.

The following proposition shows that many interpolation
 formula
 for ordinary function spaces
are also valid for mixed-norm spaces.

\begin{Proposition}\label{prop:vector a}
Let $B_0$, $B_1$ and $A$ be quasi-normed vector spaces,
all of which consist of some measurable functions defined
on $\bbR^n$ ( for $B_0$ and $B_1$ )
or on $\bbR^m$ (for $A$).
Suppose that both
$B_0$ and $B_1$ are solid.

Then for any $0<\theta<1$ and $0< p\le \infty$,
\begin{equation}\label{eq:vector}
(B_0(A), B_1(A))_{\theta,p} = (B_0, B_1 )_{\theta,p}(A)
\end{equation}
with equivalent norms.
\end{Proposition}

\begin{proof}
It suffices to show that for any $f\in B_0(A)+B_1(A)$ and $t>0$,
\[
 K(t,f;B_0(A), B_1(A))  \approx
  K(t,y\mapsto\|f(\cdot,y)\|_A; B_0, B_1) .
\]

Since $A$ is quasi-normed, there is some constant $C$ such that
for any $a_1, a_2\in A$,
\[
    \|a_1 + a_2 \|_A \le C(\|a_1\|_A + \|a_2\|_A).
\]
Suppose that $f= f_0 + f_1$, where $f_i\in B_i(A)$, $i=0,1$.
Let
\[
  g_i(y)  = \frac{\|f_i(\cdot,y)\|_A}
   {\|f_0(\cdot,y)\|_A + \|f_1(\cdot,y)\|_A}
          \|f(\cdot,y)\|_A
\]
when $\|f(\cdot,y)\|_A>0$ and $g_i(y)=0$ when $\|f(\cdot,y)\|_A = 0$, $i=0,1$.
We have $g_0(y)+g_1(y) = \|f(\cdot,y)\|_A$ and
$0\le g_i(y) \le C \|f_i(\cdot,y)\|_A$, $i=0,1$.
Since $B_i$ is solid, we have $g_i\in B_i $ and
$\|g_i\|_{B_i} \le C \|f_i\|_{B_i(A)}$, $i=0,1$.
It follows that
\[
 C  \cdot K(t,f;B_0(A), B_1(A))  \ge  K(t,y\mapsto\|f(\cdot,y)\|_A; B_0, B_1) .
\]

On the other hand, suppose that
 $\|f(\cdot,y)\|_A = g_0(y) + g_1(y)$ for some $g_i\in B_i$, $i=0,1$.
Let
\[
  f_0(x,y) = \frac{g_0(y)}{\|f(\cdot,y)\|_A} f(x,y)
     \mbox{ and }
  f_1(x,y) = \frac{g_1(y)}{\|f(\cdot,y)\|_A} f(x,y),\quad
  (x,y)\in \bbR^m\times \bbR^n.
\]
Since both $B_0$ and $B_1$ are solid, we have
\[
  \|f_0\|_{B_0(A)} =    \|g_0\|_{B_0},
   \quad  \quad
  \|f_1\|_{B_1(A)} =    \|g_1\|_{B_1}
\]
and  $f = f_0 + f_1$.
Hence
\[
 K(t,f;B_0(A), B_1(A))  \le    K(t,y\mapsto \|f(\cdot,y)\|_A; B_0, B_1) .
\]
This completes the proof.
\end{proof}

%Let $0<p_m,q_m\le \infty$ and
%$\vec {\tilde p} = (p_1,\ldots,p_{m-1})$ with $0<p_i\le \infty$,
%$1\le i\le m-1$.  Let $L^{p_m,q_m}(L^{\vec{\tilde p}})$
%denote the space consisting all measurable
%functions on $\bbR^{n_1}\times \ldots \times \bbR^{n_m}$
%such that $\left\| \| f(\cdot,x_m)\|_{L^{\vec{\tilde p}}}
%\right\|_{L_{x_m}^{p_m,q_m}}<\infty$.

The following result is an immediate consequence of
Propositions
 \ref{prop:interpolation Lorentz a}
and \ref{prop:vector a}.

\begin{Corollary}\label{Co:Lorentz}
Suppose that
$\vec u = (u_1,\ldots,u_m)$
  with
$0<u_i \le \infty$, $1\le i\le m$.
Let   $0<p_0$, $p_1$, $q_0$, $q_1$,
 $q\le \infty$,
$0<\theta<1$ and $1/p = (1-\theta)/p_0+\theta/p_1$.
Then for $p_0\ne p_1$, we have
\[
  (L^{p_0,q_0}(L^{\vec u}) , L^{p_1,q_1}(L^{\vec u}) )_{\theta,q}
  = L^{p,q}(L^{\vec u})
\]
with equivalent norms.
And the formula is also true for $p_0=p_1$
provided $1/q = (1-\theta)/q_0+\theta/q_1$.
\end{Corollary}

Next we give an  interpolation
formula for mixed-norm  Lebesgue spaces.

\begin{Theorem}
\label{thm:mixed Lp}
Suppose that
$0<\theta<1$,
$0<u,v\le \infty$ and $\vec p=(p_1,\ldots,p_m)$ with
  $0<p_i\le \infty$, $1\le i\le m$.

For any $0<q\le \min\{p_i:\, 1\le i\le m\}$,
we have
\begin{equation}\label{eq:interpolation:m}
  (L^{\vec p}(L^{u}), L^{\vec p}(L^{v}))_{\theta,q}
   \hookrightarrow
   L^{\vec p}((L^u,L^v)_{\theta,q}).
\end{equation}

Conversely,  if   $ p_i  \le q< \infty$, $ 1\le i\le m$,
then
\begin{equation}\label{eq:interpolation:m2}
(L^{\vec p}(L^{u}), L^{\vec p}(L^{v}))_{\theta,q}
   \hookleftarrow
   L^{\vec p}((L^u,L^v)_{\theta,q}).
\end{equation}

In particular,
whenever $0<q= p_1\!=\ldots=p_m\!<\infty$, we have
$
%\begin{equation}\label{eq:e:p:m}
( L^{\vec p}(L^{u}), L^{\vec p}(L^{v}))_{\theta,q}
$ $=     L^{\vec p}((L^u,L^v)_{\theta,q})
$
%\end{equation}
with equivalent norms.
\end{Theorem}

Recall that mixed-norm and Bochner spaces
contain functions with different measurability.
So
Theorem~\ref{thm:mixed Lp} is an analog,
but not a consequence, of
Proposition~\ref{prop:interpolation pq a}.

To prove Theorem~\ref{thm:mixed Lp}, we need to show
that $\|f(\cdot,y)\|_{(L^u,L^v)_{\theta,q}}$ is a measurable
function of $y$ whenever $f $ is   measurable.
Although  $\|f(\cdot,y)\|_{L^{p_{\theta},q}}$ is measurable
and $\|f(\cdot,y)\|_{(L^u,L^v)_{\theta,q}}$
$
\approx \|f(\cdot,y)\|_{L^{p_{\theta},q}}$,
where $1/p_{\theta}= (1-\theta)/u+\theta/v$,
the measurability of $\|f(\cdot,y)\|_{(L^u,L^v)_{\theta,q}}$
is not obvious.

It is well known that for a normed space,
the convergence $\|x_n-x\|\rightarrow 0$ implies
$\|x_n\|\rightarrow \|x\|$.
We show that the same is true for
the quasi-norm
$f\mapsto K(t,f;L^u,L^v)$.

\begin{Lemma}\label{Lm:quasi-norm}
Let $0<u,v\le \infty$.
Suppose that $f,f_n\in L^u+L^v$
and for some $t_0>0$, $\lim_{n\rightarrow\infty}K(t_0,f_n-f;L^u,L^v)=0$.
Then for any $t>0$, $\lim_{n\rightarrow\infty}K(t ,f_n-f;L^u,L^v)=0$
and $\lim_{n\rightarrow\infty}K(t,f_n;L^u,L^v)=K(t,f;L^u,L^v)$.
\end{Lemma}

\begin{proof}
For brevity, we write $K(t,f)$ instead of $K(t,f;L^u,L^v)$.
We see from (\ref{eq:ec:e31}) and the hypothesis that
for any $t>0$,
$\lim_{n\rightarrow\infty}K(t ,f_n-f)=0$.
Whenever $u,v\ge 1$, $K(t,f)$ is sub-additive
and therefore the conclusion is obvious.

In the followings  we only consider the case
$u,v<1$. The cases $u<1\le v$ and $v<1\le u$ can be proved similarly.

Fix some $t>0$. since $K(t,f_n)\le C(K(t,f)+K(t,f_n-f))$, we have
\[
 M:=\sup_{n\ge 1} K(t,f_n)<\infty.
\]
For any $ \varepsilon>0$ and $n\ge 1$,
there exist some $f_{n,0}\in L^u$
and $f_{n,1}\in L^v $ such that $f_n = f_{n,0}+f_{n,1}$ and
\begin{equation}\label{eq:ec:e33}
  \|f_{n,0}\|_{L^u} + t\|f_{n,1}\|_{L^v} \le (1+\varepsilon)K(t,f_n).
\end{equation}
And for $n$ large enough,  there exist
some $h_{n,0}\in L^u$
and $h_{n,1}\in L^v $ such that  $f-f_n = h_{n,0}+h_{n,1}$ and
\[
  \|h_{n,0}\|_{L^u} + t\|h_{n,1}\|_{L^v} < \varepsilon.
\]
Note that
\begin{align}
\|f_{n,0} + h_{n,0}\|_{L^u} - \|f_{n,0}\|_{L^u}
&\le (\|f_{n,0}\|_{L^u}^u + \|h_{n,0}\|_{L^u}^u)^{1/u}
-
\|f_{n,0}\|_{L^u}. \label{eq:ec:e32}
\end{align}
By Lagrange's mean value theorem, there exists some $\xi$ between
$\|f_{n,0}\|_{L^u}^u + \|h_{n,0}\|_{L^u}^u$
and $\|f_{n,0}\|_{L^u}^u $ such that
\begin{align}
  (\|f_{n,0}\|_{L^u}^u + \|h_{n,0}\|_{L^u}^u)^{1/u}
-
\|f_{n,0}\|_{L^u}
&= \frac{1}{u} \xi^{1/u-1} \|h_{n,0}\|^u_{L^u} \nonumber \\
&\le \frac{1}{u} (\|f_{n,0}\|_{L^u}^u + \|h_{n,0}\|_{L^u}^u)^{1/u-1} \|h_{n,0}\|^u_{L^u}
  \nonumber \\
&\le\frac{1}{u} ((1+\varepsilon)^uM^u+\varepsilon^u)^{1/u-1} \varepsilon^u.  \nonumber
\end{align}
By (\ref{eq:ec:e32}), we get
\[
  \|f_{n,0} + h_{n,0}\|_{L^u} - \|f_{n,0}\|_{L^u}
  \le \frac{1}{u} ((1+\varepsilon)^uM^u+\varepsilon^u)^{1/u-1} \varepsilon^u.
\]
Similarly,
\[
  t(\|f_{n,1} + h_{n,1}\|_{L^v} - \|f_{n,1}\|_{L^v})
  \le \frac{t}{v} ((1+\varepsilon)^vM^v+\varepsilon^v)^{1/v-1} \varepsilon^v.
\]
Hence
\begin{align*}
K(t,f)
&\le \|f_{n,0} + h_{n,0}\|_{L^u} +t(\|f_{n,1} + h_{n,1}\|_{L^v}) \\
&\hskip-8mm\le \|f_{n,0}\|_{L^u} + t\|f_{n,1}\|_{L^v}
   + \frac{1}{u} ((1+\varepsilon)^uM^u\!+\!\varepsilon^u)^{\frac{1}{u}-1} \varepsilon^u
   + \frac{t}{v} ((1+\varepsilon)^vM^v\!+\!\varepsilon^v)^{\frac{1}{v}-1} \varepsilon^v.
\end{align*}
Letting $n\rightarrow\infty$ and $\varepsilon\rightarrow 0$ successively, we see from (\ref{eq:ec:e33}) that
\[
K(t,f)
\le
  \liminf_{n\rightarrow\infty}   K(t,f_n).
\]
Similar arguments show that $  \limsup_{n\rightarrow\infty}   K(t,f_n)\le K(t,f)$. Now we get the conclusion as desired.
\end{proof}

The following is a representation of the $K$-functional.

%\begin{Lemma}\label{Lm:K-representation}
%For any $0<u,v\le \infty$, $t>0$ and   measurable function $f$,
%\[
%  K(t,f;L^u,L^v) = \inf\{\|f_0\|_{L^u}+t\|f_1\|_{L^v}:\,
%f=f_0+f_1, f_0 \bar f_1=|f_0f_1| \},
%\]
%where $\bar f$ is the complex conjugate of $f$.
%As a result, $K(t,f;L^u,L^v) = K(t,|f|;L^u,L^v)$.
%\end{Lemma}
%
%\begin{proof}
%Suppose that $f=f_0+f_1$.
%For $i=0,1$, let
%\[
%\tilde f_i(x) =
%   \begin{cases}
%   \frac{f(x)\cdot |f_i(x)| }{ |f_0(x)|+|f_1(x)| },
% & \mbox{ if } f(x)\ne 0, \\
 %  0, & \mbox{ otherwise}.
%   \end{cases}
%\]
%Then we have $f = \tilde f_0 + \tilde f_1$,
%$\tilde f_0 \bar{\tilde f}_1 = |\tilde f_0 \bar{\tilde f}_1|$
%and $|\tilde f_i| \le |f_i|$, $i=0,1$.
%Hence $\|\tilde f_0\|_{L^u}+t\|\tilde f_1\|_{L^v}\le
%\|f_0\|_{L^u}+ t\|f_1\|_{L^v}$.
%Now the conclusion follows.
%\end{proof}

\begin{Lemma}\label{Lm:K-representation}
For any $0<u,v\le \infty$, $t>0$ and   measurable function $f$,
\begin{align*}
  K(t,f;L^u,L^v)
  &= \inf\{\|f_0\|_{L^u}+t\|f_1\|_{L^v}:\,
 f=f_0+f_1, f_0 \bar f_1=|f_0f_1| \} \\
 & = \inf\{\|c(x)f(x)\|_{L^u}+t\|(1-c(x))f(x)\|_{L^v}:\,
  0\le c(x)\le 1 \}.
\end{align*}
where $\bar f$ is the complex conjugate of $f$.
As a result, $K(t,f;L^u,L^v) = K(t,|f|;L^u,L^v)$.
\end{Lemma}

\begin{proof}
Suppose that $f=f_0+f_1$.
Let
\begin{align*}
 c(x) &=\begin{cases}
   \frac{|f_0(x)|}{|f_0(x)| + |f_1(x)|},
    &   \mbox{ if } f(x)\ne 0, \\
   0, & \mbox{ otherwise},
 \end{cases}\\
\tilde f_0(x) &= c(x) f(x)
\mbox{\quad and \quad }
\tilde f_1(x) = (1-c(x))f(x).
\end{align*}
Then we have $f = \tilde f_0 + \tilde f_1$,
$\tilde f_0 \bar{\tilde f}_1 = |\tilde f_0 \bar{\tilde f}_1|$
and $|\tilde f_i| \le |f_i|$, $i=0,1$.
Hence $\|\tilde f_0\|_{L^u}+t\|\tilde f_1\|_{L^v}\le
\|f_0\|_{L^u}+ t\|f_1\|_{L^v}$.
Now the conclusion follows.
\end{proof}

Next we show that the monotone convergence theorem is true
on $L^u+L^v$.

\begin{Lemma} \label{Lm:monotone}
Suppose that $0<u,v\le \infty$.
Let
$\{f_n(x):\, n\ge 1\}$ be a sequence of non-negative
measurable functions which are increasing almost everywhere.
Then for any $t>0$,
\[
  \lim_{n\rightarrow\infty} K(t,f_n;L^u,L^v)
   =
   K(t,\lim_{n\rightarrow\infty} f_n;L^u,L^v).
\]
As a consequence, the monotone convergence theorem is true
on $L^u+L^v$.
\end{Lemma}

\begin{proof}
Denote $f =\lim_{n\rightarrow\infty}  f_n$.
Since $\{f_n:\, n\ge 1\}$ is increasing, we have
\[
 M:= \lim_{n\rightarrow\infty} K(t,f_n;L^u,L^v)
   \le
   K(t,f;L^u,L^v).
\]
So it suffices to show that $K(t,f;L^u,L^v)\le M$,
for which we only need to consider the case
$M<\infty$.

Observe that
\[
  K(t,f ;L^u,L^v) = t K(\frac{1}{t},f; L^v, L^u).
\]
Without loss of generality, we may further assume that
$u\le v$.
For brevity, we write $K(t,f )$ instead of
$K(t,f ;L^u,L^v)$.

First, we consider the case $v<\infty$.
By Lemma~\ref{Lm:K-representation}, for each $n\ge 1$, there exist some
$f_{n,0}\in L^u$
and $f_{n,1}\in L^v$ such that $f_n =f_{n,0}+f_{n,1}$,
$f_{n,0}, f_{n,1}\ge 0$ and
\begin{equation}\label{eq:ec:e23}
  \|f_{n,0}\|_{L^u} + t\|f_{n,1}\|_{L^v}
  \le 2K(t,f_n).
\end{equation}

Observe that
\begin{align*}
f_n \chi^{}_{\{ f_n > 2\}}
&= f_{n,0} \chi^{}_{\{ f_n > 2\}}
  +f_{n,1} \chi^{}_{\{  f_{n,1}  \le 1,  f_n > 2\}}
   +f_{n,1} \chi^{}_{\{ f_{n,1} > 1,  f_n > 2\}}  \\
&\le  2 f_{n,0} \chi^{}_{\{f_n> 2\}}
    +f_{n,1} \chi^{}_{\{ f_{n,1} > 1,  f_n > 2\}} .
\end{align*}
We have
\begin{align*}
\| f_n \chi^{}_{\{ f_n > 2\}}\|_{L^u}
&\lesssim \|
f_{n,0} \chi^{}_{\{ f_n > 2\}}\|_{L^u}
    +\|f_{n,1} \chi^{}_{\{ f_{n,1} > 1,  f_n > 2\}}\|_{L^u}
      \\
&\lesssim
  \|
f_{n,0}  \|_{L^u}
    +\|f_{n,1} \|_{L^v}^{v/u}\\
&\le 2M + \Big(\frac{2M}{t}\Big)^{v/u}.
\end{align*}
On the other hand,
\begin{align*}
\| f_n \chi^{}_{\{f_n\le 2\}}\|_{L^v}
&\lesssim \|f_{n,0} \chi^{}_{\{f_n\le 2\}}\|_{L^v}
   + \|f_{n,1}\chi^{}_{\{f_n\le 2\}} \|_{L^v} \\
&
  \lesssim   \|f_{n,0}  \|_{L^u}^{u/v}
   + \|f_{n,1}\|_{L^v}\\
&\le (2M)^{u/v} + \frac{2M}{t}  .
\end{align*}
Hence
\begin{align*}
& \| f_n \chi^{}_{\{f_n> 2\}}\|_{L^u}
 +t\| f_n \chi^{}_{\{f_n\le 2\}}\|_{L^v}
 \le C_{u,v,t,M},\quad n\ge 1 .
\end{align*}
Since $\{f_n:\, n\ge 1\}$ is increasing,
$ \{f \le 2\}\subset \{f_n\le 2\}$ for all $n\ge 1$. Hence
\begin{align*}
&\| f_n \chi^{}_{\{f_n> 2\}}\|_{L^u}
 +t\| f_n \chi^{}_{\{f\le 2\}}\|_{L^v}
   \le  C_{u,v,t,M},\quad n\ge 1 .
\end{align*}
Letting $n\rightarrow\infty$, we see from the monotone convergence
theorem that
\[
  \|f \chi^{}_{\{f > 2\}}\|_{L^u}
+
t\|f \chi^{}_{\{f \le 2\}}\|_{L^v}
<\infty.
\]
Consequently, $f\in L^u+L^v$.

Set $f_u = f \chi^{}_{\{f > 2\}}$
and $f_v = f \chi^{}_{\{f \le 2\}}$.
Then $f_u\in L^u$,
$f_v\in L^v$
and $f=f_u+f_v$. Hence
\begin{align*}
K(t,f-f_n )
&\le \|f_u - \frac{f_u}{f}f_n\|_{L^u}
   + t\|f_v - \frac{f_v}{f}f_n\|_{L^v}
\rightarrow 0,\quad n\rightarrow\infty,
\end{align*}
where we use Lebesgue's dominated convergence theorem and the hypothesis
that $f_n\rightarrow f$.
Now we see from Lemma~\ref{Lm:quasi-norm} that
 $\lim_{n\rightarrow\infty}K(t,f_n ) =
K(t,f )$.

Next we consider the case $0<u<v=\infty$.
For any $\varepsilon>0$ and $n\ge 1$,
since $K(t,f_n)\le M$, there exist some $f_{n,0}\in L^u$
and $f_{n,1}\in L^{\infty}$ such that
$f_n=f_{n,0}+f_{n,1}$ and
\begin{equation}\label{eq:ec:e20}
  \|f_{n,0}\|_{L^u} + t\|f_{n,1}\|_{L^{\infty}} \le M+\varepsilon.
\end{equation}
%
%If there is some integer $N$ such that
%$\|f_{n,1}\|_{L^{\infty}}=0$ for all $n\ge N$,
%then $f_{n,0}=f_n$, a.e.
%Hence
%\[
%  K(t,f) \le \|f\|_{L^u}
%   = \lim_{n\rightarrow\infty}\|f_n\|_{L^u}
%     = \lim_{n\rightarrow\infty}\|f_{n,0}\|_{L^u}
%\le M+\varepsilon.
%\]
%Letting $\varepsilon\rightarrow 0$, we get $K(t,f)\le M$.
%
%Now assume that there are infinitely many $n$ such that
%$\alpha_n:= \|f_{n,1}\|_{L^{\infty}}>0$.
%Without loss of generality, we assume that
%$\alpha_n>0$ for all $n\ge 1$.
Denote $\alpha_n = \|f_{n,1}\|_{L^{\infty}}$. Observe that
\[
 0\le (f_n -\alpha_n) \chi^{}_{\{f_n>\alpha_n\}}
  =(f_{n,0}+f_{n,1} -\alpha_n) \chi^{}_{\{f_n>\alpha_n\}}
  \le f_{n,0} \chi^{}_{\{f_n>\alpha_n\}},\qquad a.e.
\]
It follows from (\ref{eq:ec:e20}) that
\begin{align}
\|  (f_n -\alpha_n) \chi^{}_{\{f_n>\alpha_n\}}\|_{L^u}
  + t\alpha_n \le M+\varepsilon,\qquad n\ge 1. \label{eq:ec:e21}
\end{align}
Hence $\{\alpha_n:\, n\ge 1\}$ is a bounded sequence.
Consequently, there is a convergent subsequence.
Without loss of generality, assume that
$\alpha_n\rightarrow \alpha$ as $n\rightarrow \infty$.

For any $k\ge 1$, there is some $n_k$ such that
\[
  |\alpha_{n_k} - \alpha | \le \frac{1}{k}.
\]
For any $k\ge k_0\ge 1$,
\begin{align*}
 \Big\| (f_{n_k} - \alpha - \frac{1}{k} )
\chi^{}_{\{f_{n_k} > \alpha+1/k_0 \}}\Big\|_{L^u}
  + t\alpha_{n_k}
&\le
\Big\| (f_{n_k} - \alpha_{n_k}  )
\chi^{}_{\{f_{n_k} \ge \alpha_{n_k}\}}\Big\|_{L^u}
  + t\alpha_{n_k}\\
&\le    M+\varepsilon ,
\end{align*}
where we use (\ref{eq:ec:e21}) in the last step.
Letting $k\rightarrow\infty$ and
$k_0\rightarrow\infty$ successively, we see from  the monotone convergence
theorem that
\[
  \Big\|  (f - \alpha)
\chi^{}_{\{f >  \alpha \}}\Big\|_{L^u}
  + t \alpha
  \le  M+\varepsilon.
\]
Hence
\begin{align*}
K(t,f)
 &\le \big\|  (f - \alpha)
\chi^{}_{\{f> \alpha \}}\big\|_{L^u}
  + t \big\|  f
\chi^{}_{\{f\le  \alpha \}}
  + \alpha\chi^{}_{\{f>\alpha \}}
\big\|_{L^{\infty}}  \\
&\le  \big\|  (f - \alpha)
\chi^{}_{\{f> \alpha \}}\big\|_{L^u} + t \alpha   \\
&\le   M+\varepsilon.
\end{align*}
Since $\varepsilon$ is arbitrary, we get $K(t,f)\le M$.

Finally, we consider the case $u=v=\infty$.
In this case,
$K(t,f) \!=\! \min\{1,t\}\|f\|_{L^{\infty}}$.
Hence $\lim_{n\rightarrow\infty} K(t,f_n) =
K(t,f)$. This completes the proof.
\end{proof}

It is well known that if $f(x,y)$ is a measurable function
and $0<p\le\infty$,
then $\|f(\cdot,y)\|_{L^p}$ is measurable as a function of
$y$. We show that the same is true whenever
the $L^p$ norm is replaced by some other quasi-norms.

\begin{Lemma}\label{Lm:measurability}
Let $f(x,y)$ be a   measurable function on
$\bbR^m\times \bbR^n$ such that
$|f(x,y)| $  is finite for almost all $(x,y)$.
Suppose that $0<u,v\le \infty$.
\begin{enumerate}
\item For any $t>0$, $K(t,f(\cdot,y);L^u,L^v)$ is measurable as a function of $y\in\bbR^n$.

\item $K(t,f(\cdot,y);L^u,L^v)$ is measurable
as a function of $(t,y)\in (0,\infty)\times \bbR^n$.

\item
For any $0<\theta<1$
 and $0<p\le \infty$,
$ \|f(\cdot,y)\|_{(L^u,L^v)_{\theta,p}}$
is  measurable  on $\bbR^n$.
\end{enumerate}
\end{Lemma}

\begin{proof}
We write $K(t,f(\cdot,y))$ instead of
$K(t,f(\cdot,y);L^u,L^v)$ for brevity.
It suffices to prove the first two items
since
(iii) is a consequence of (ii).
%Moreover, we only need to consider non-negative functions
%$f$, thanks to Lemma~\ref{Lm:K-representation}.

(i)\,\,
If $u=v=\infty$, then $K(t,f(\cdot,y)) =\min\{1,t\}\|f(\cdot,y)\|_{L^{\infty}}$,
which is measurable with respect to $(t,y)$.
It remains to consider the case $\min\{u,v\}<\infty$.
Without loss of generality, assume that $u<\infty$.

First, we assume that $f$ is a compactly supported
bounded function. Take some constant $A>0$ such that
$|f(x,y)|\le A$ for all $(x,y)\in \bbR^m\times\bbR^n$
and
$f(x,y)=0$ whenever $|x|+|y|>A$.
Then $f\in L^1(L^u)$.
By Proposition~\ref{prop:density},
there is a sequence
$\{f_k:\, k\ge 1\}\subset L^1(L^u )$ of the form
\[
  f_k(x,y) = \sum_{i=1}^{r} a_i(x) \chi^{}_{E_i}(y),
\]
where $a_i\in L^u $,
$E_i$ are  measurable sets in $\bbR^n$
which are pairwise disjoint and of finite measures,
$1\le i\le r$, and $r$ is a positive integer,
such that
\[
  \lim_{k\rightarrow \infty} \Big \| \|f_k(x,y) - f(x,y)\|_{L_x^u}\Big\|_{L_y^1}
  =
  \lim_{k\rightarrow \infty} \|f_k - f\|_{L^1(L^u )}
  =0.
\]
%Recall that the quasi-norm on $L^u\cap L^v$
%is defined by $\|\cdot\|_{L^u\cap L^v} = \max\{\|\cdot\|_{L^u},
%\|\cdot\|_{L^v}\}$.
Hence there is a subsequence of $\{f_k:\, k\ge 1\}$,
say itself, such that for almost all $y\in\bbR^n$,
\[
 \lim_{k\rightarrow\infty} \|f_k(x,y) - f(x,y)\|_{L_x^u}=0.
\]
Consequently, for almost all $y\in\bbR^n$,
\begin{align*}
K(t,f(\cdot,y)-f_k(\cdot,y) )
\le \|f_k(x,y) - f(x,y)\|_{L_x^u} \rightarrow 0.
\end{align*}
It follows  from Lemma~\ref{Lm:quasi-norm} that
\[
  K(t,f(\cdot,y) )
  =\lim_{k\rightarrow\infty }  K(t,f_k(\cdot,y) ).
\]
Since for any $t>0$ and $k\ge 1$,
$K(t,f_k(\cdot,y) )
=\sum_{i=1}^r K(t,a_i )\cdot \chi^{}_{E_i}(y)$
is a measurable function on $\bbR^n$,
$K(t,f(\cdot,y)  )$
is also  measurable   on $\bbR^n$ for any $t>0$.

For the general case, set
\[
  f_k = f\cdot \chi^{}_{\{|x|,|y|\le k\}} \chi^{}_{\{|f|\le k\}}.
\]
Then the sequence $\{|f_k|:\, k\ge 1\}$ is convergent to
$|f|$ increasingly almost everywhere. By Lemmas~\ref{Lm:K-representation} and \ref{Lm:monotone},
for any $t>0$,
\[
  K(t,f(\cdot,y)) =
  K(t,|f(\cdot,y)|) = \lim_{k\rightarrow\infty} K(t,|f_k(\cdot,y)|),
  \quad a.e.\ y\in\bbR^n.
\]
Hence $K(t,f(\cdot,y) )$ is  measurable as a  function
of  $y\in \bbR^n$.

(ii)\,\,
By (\ref{eq:ec:e31}), $K(t,f(\cdot,y) ) $ is continuous with respect to
$t$. Hence for any $\alpha\in\bbR$,
\[
  \{(t,y):\, K(t,f(\cdot,y) )>\alpha\}
   = \bigcup_{r\in\mathbb Q^+}
     [r,\infty)\times \{y:\,  K(r,f(\cdot,y) )>\alpha\},
\]
where $\mathbb Q^+$ is the set of positive rational numbers.
Hence $K(t,f(\cdot,y) )$ is a measurable function
on $(0,\infty)\times \bbR^n$.
This completes the proof.
\end{proof}

We are now ready to give a proof of  Theorem~\ref{thm:mixed Lp}.

\begin{proof} [Proof of Theorem~\ref{thm:mixed Lp}]
Denote functions in $L^{\vec p}(L^u)$ by $f(x,y)$,
where $x\in\bbR^{n_0}$, $y=(y_1,\ldots,y_m)\in\bbR^{n_1+\ldots+n_m}$.

First, we  prove (\ref{eq:interpolation:m}).
Note that for any $0<s\le \infty$ and non-negative  functions $g,h\in L^s(\bbR^n)$,
\[
 \| g+h\|_{L^s} \approx \|g\|_{L^s} + \|h\|_{L^s} .
\]
For any $f  \in  (L^{\vec p}(L^{ u}),  L^{\vec p}(L^{v} ))_{\theta,q} $, we have
\begin{align*}
 \|f\|_{(L^{\vec p}(L^{ u}),  L^{\vec p}(L^{v} ))_{\theta,q}}
\hskip -40pt&\hskip 40pt
 = \left\|\{2^{-n\theta} K(2^n,f; L^{\vec p}(L^u),
   L^{\vec p}(L^v))\} \right\|_{\ell^q}\\
&= \left\|\Big\{2^{-n\theta}\inf\{ \|f_0\|_{L^{\vec p}(L^u)}
   + 2^n \|f_1\|_{L^{\vec p}(L^v)}:\, f=f_0+f_1 \} \Big\} \right\|_{\ell^q}\\
&\approx\left\|\Big\{2^{-n\theta}\inf\{ \big\|
  \|f_0(\cdot,y)\|_{L^u}
   + 2^n \|f_1(\cdot,y)\|_{L^v}\big\|_{L_y^{\vec p}}:\, f=f_0+f_1 \} \Big\} \right\|_{\ell^q}\\
&\ge \left\|\Big\{2^{-n\theta}
   \| K(2^n, f(\cdot,y); L^u,L^v)\|_{L_y^{\vec p}}
 \Big\} \right\|_{\ell^q}.
\end{align*}
Since $q\le \min\{p_i:\, 1\le i\le m\}$, it follows from
Minkowski's inequality that
\begin{align*}
  \left\|\Big\{2^{-n\theta}
   \| K(2^n, f(\cdot,y); L^u,L^v)\|_{L_y^{\vec p}}
 \Big\} \right\|_{\ell^q}
 &\ge  \left\| \big\| \big\{2^{-n\theta}
    K(2^n, f(\cdot,y); L^u,L^v)\big\} \big\|_{\ell^q}
        \right\|_{L_y^{\vec p}}\\
 &=\Big\|\|f(\cdot,y)\|_{(L^u,L^v)_{\theta,q}} \Big\|_{L^{\vec p}}.
\end{align*}
Hence
$
  \|f\|_{(L^{\vec p}(L^{ u}),  L^{\vec p}(L^{v} ))_{\theta,q}}
  \gtrsim
  \|f \|_{L^{\vec p}((L^u,L^v)_{\theta,q})}$.
This proves (\ref{eq:interpolation:m}).

Next we prove (\ref{eq:interpolation:m2}).
If $u=v=\infty$,
then $L^{\vec p}((L^u,L^v)_{\theta,q})
= L^{\vec p}(L^{\infty}) = (L^{\vec p}(L^u), L^{\vec p}(L^v))_{\theta,q}$.
Hence (\ref{eq:interpolation:m2}) is true.

Now we consider the case $\min\{u,v\}<\infty$.
%Since $\|\cdot\|_{(L^u,L^v)_{\theta,q}}
%\approx \|\cdot\|_{L^{p_0,q}}$, where $1/p_0 = %(1-\theta)/u+\theta/v$,
%$(L^u,L^v)_{\theta,q}$ is separable whenever
%$0<u,v,q<\infty$. Consequently, there is a dense subset   %$\{a_i:\, i\ge 1\}$ of $(L^u,L^v)_{\theta,q}$.

Suppose that  $f(x,y) = \sum_{i=1}^k a_i(x) \chi^{}_{E_i}(y)
\in L^{\vec p}((L^u,L^v)_{\theta,q})$,
where  $a_i\in (L^u,L^v)_{\theta,q}$,
$E_i$ are pairwise disjoint measurable subsets of $\bbR^{n_1+\ldots+n_m}$
and
 $|E_i|<\infty$,
$1\le i\le k$.

For any $\varepsilon>0$ and $n\in\bbZ$,
there exist $a_{i,u}^{(n)}\in L^{u}$
and $a_{i,v}^{(n)}\in L^{v}$
such that $a_i = a_{i,u}^{(n)} + a_{i,v}^{(n)}$ and
\[
    \| a_{i,u}^{(n)}\|_{L^{u}}
  +2^n \| a_{i,v}^{(n)}\|_{L^{v}}
  \le (1+\varepsilon) K(2^n, a_i; L^{u}, L^{v}).
\]
Set
\[
 f_{n,u}(x,y) = \sum_{i=1}^k a_{i,u}^{(n)}(x) \chi^{}_{E_i}(y),
 \qquad
 f_{n,v} (x,y)= \sum_{i=1}^k a_{i,v}^{(n)} (x)\chi^{}_{E_i}(y).
\]
We have $f= f_{n,u}+f_{n,v}$, $n\in\bbZ$.
It follows  that
\begin{align}
  K(2^n,f;L^{\vec p}(L^{u}), L^{\vec p}(L^{v}))
&\le
\|f_{n,u}\|_{L^{\vec p}(L^{u})}
  + 2^n \|f_{n,v}\|_{L^{\vec p}(L^{v})} \nonumber \\
&\approx \Big \| \|f_{n,u}(\cdot,y)\|_{L^u}
   +  2^n\|f_{n,v}(\cdot,y)\|_{L^v}\Big\|_{L^{\vec p}} \nonumber \\
&=  \bigg\|\sum_{i=1}^k \Big (\|a_{i,u}^{(n)}\|_{L^{u}}
+2^n \| a_{i,v}^{(n)}\|_{L^{v}}\Big)
   \chi^{}_{E_i}(y)  \bigg\|_{L^{\vec p}} \nonumber \\
&\le (1+\varepsilon)   \bigg\|\sum_{i=1}^k  K(2^n, a_i; L^{u}, L^{v})
   \chi^{}_{E_i}(y)  \bigg\|_{L^{\vec p}}  \nonumber \\
&= (1+\varepsilon)    \| K(2^n, f(\cdot,y); L^{u}, L^{v})
    \|_{L^{\vec p}}. \label{eq:ec:e36}
\end{align}
Hence
\begin{align*}
\|f\|_{(L^{\vec p}(L^{u}),L^{\vec p}(L^{v}))_{\theta,q}}
&=
 \bigg(\sum_{n\in\bbZ} \Big( 2^{-n\theta} K(2^n,f;L^{\vec p}(L^{u}), L^{\vec p}(L^{v}))\Big)^q
 \bigg)^{1/q} \\
& \lesssim (1+\varepsilon)
\bigg(\sum_{n\in\bbZ} \Big( 2^{-n\theta} \| K(2^n, f(\cdot,y); L^{u}, L^{v})
    \|_{L^{\vec p}} \Big)^q
 \bigg)^{1/q} .
 \end{align*}
Since $p_i\le q$ for all $1\le i\le m$, we see from Minkowski's inequality that
\begin{align}
\|f\|_{(L^{\vec p}(L^{u}),L^{\vec p}(L^{v}))_{\theta,q}}
&\lesssim (1+\varepsilon)
 \bigg\|\bigg( \sum_{n\in\bbZ} \Big( 2^{-n\theta}
  K(2^n, f(\cdot,y); L^{u}, L^{v})
      \Big)^q
 \bigg)^{1/q}\bigg \|_{L^{\vec p}}\nonumber \\
&= (1+\varepsilon)   \|f\|_{L^{\vec p}((L^u,L^v)_{\theta,q})}.\nonumber
 \end{align}
Note that $(L^u,L^v)_{\theta,q} = L^{p_{\theta},q}$
is separable, where $1/p_{\theta} = (1-\theta)/u+\theta/v$.
By Proposition~\ref{prop:density},
functions of the form
$\sum_{i=1}^k a_i (x)\chi^{}_{E_i}(y)$ are dense in
$L^{\vec p}((L^u,L^v)_{\theta,q}) $. It follows from the above inequality that  (\ref{eq:interpolation:m2}) is true.
\end{proof}

%More results on interpolation spaces  can be found in
%\cite{Bergh1976,Hytonen2016} and references therein.

\section{Proof of The Main Results}

It is well known that for any $f\in L^p$ with $1\le p<\infty$, $\lim_{|y|\rightarrow\infty} \|f+f(\cdot-y)\|_p=2^{1/p} \|f\|_p$.
For Lebesgue spaces with mixed norms, we show that the limit is path dependent.

\begin{Lemma}[{\cite[Lemma 2.3]{ChenSun2019}}]\label{Lm: translation}
Let $y = (y_1,\ldots,y_m)\in \bbR^{N_m}$. Suppose that
$y_k\ne 0$
for some $1\le k\le m$  and $y_i=0$ for all $k+1\le i\le m$.     Then for any $f\in L^{\vec p}(\bbR^{N_m})$, where $\vec p=(p_1, \ldots, p_m)$ with $1\le p_i<\infty$, $1\le i\le m$, we have
\[
  \lim_{|y_k|\rightarrow \infty} \| f(\cdot -  y) + f\|_{L^{\vec p}}
  =
  2^{1/p_k}\|f\|_{L^{\vec p}}.
\]
\end{Lemma}

The following lemma gives a testing condition
for the boundedness of $I_{\lambda}$
when $p_m=1$.

\begin{Lemma}\label{Lm:end points}
Suppose that $\vec p = (p_1,\ldots,p_m)$
and $\vec q=(q_1,\ldots,q_m)$ with
$1\le p_i, q_i\le \infty$ %for  $1\le i\le m$
satisfy (\ref{eq:homo:m}), $1\le i\le m$.
If there is some $1\le i_1\le m-1$ such that $p_{i_1+1}=\ldots=p_m=1$, then $I_{\lambda}$ is bounded from $L^{\vec p}$ to $L^{\vec q}$ if and only if
\begin{equation}\label{eq:e2}
  \Big \| \int_{\bbR^{N_{i_1}}}
     \frac{f(y_1,\ldots,y_{i_1}) dy_1\ldots dy_{i_1}}{(\sum_{i=1}^{i_1}
     |x_i-y_i| + \sum_{i={i_1+1}}^m |x_i|)^{\lambda}}
     \Big\|_{L^{\vec q}}
     \le C_{\vec p, \vec q} \|f\|_{L^{\vec {\tilde p}}},
\end{equation}
where  $\vec{\tilde p}=(p_1,\ldots,p_{i_1})$ and $N_{i_1} = n_1+\ldots+n_{i_1}$.

When the condition is satisfied, we have
$
  \|I_{\lambda}\|_{L^{\vec p}\rightarrow L^{\vec q}}
  \le C_{\vec p, \vec q}.
$
\end{Lemma}

\begin{proof}
First, we assume that $I_{\lambda}$ is bounded.
Let
\[
  f(y_1,\ldots,y_m)
  = \tilde f(y_1,\ldots,y_{i_1}) \prod_{i=i_1+1}^m
  \frac{1}{\delta^{n_i}} \chi^{}_{\{|y_i|\le \delta\}}(y_i),
\]
where $\tilde f \in L^{\vec{\tilde p}}$ is non-negative.
We have
\[
    \Big \| \int_{\bbR^{N_m}}
     \frac{f(y_1,\ldots,y_m) dy_1\ldots dy_m}{(\sum_{i=1}^m
     |x_i-y_i|  )^{\lambda}}
     \Big\|_{L^{\vec q}}
     \le  \|I_{\lambda}\|_{L^{\vec p}\rightarrow L^{\vec q}} \|f\|_{L^{\vec p}}.
\]
By the mean value theorem, there exist some
$\xi_i\in\bbR^{n_i}$ with $|\xi_i|\le \delta$ for $i_1+1\le i\le m$
such that
\[
    \Big\| \int_{\bbR^{n_1+\ldots+n_{i_1}}}
     \frac{\tilde{f}(y_1,\ldots,y_{i_1}) dy_1\ldots dy_{i_1}}{(\sum_{i=1}^{i_1}
     |x_i-y_i| + \sum_{i={i_1+1}}^m |x_i-\xi_i| )^{\lambda}}
     \Big\|_{L^{\vec q}}
     \le   \|I_{\lambda}\|_{L^{\vec p}\rightarrow L^{\vec q}} \|\tilde f\|_{L^{\vec {\tilde p}}}.
\]
Letting $\delta\rightarrow 0$, we see from Fatou's lemma that
(\ref{eq:e2}) is true.

Next we assume that (\ref{eq:e2}) is true.
Then for any $f\in L^{\vec {\tilde p}}$ and $g\in L^{\vec q\ \!\!'}$,
\[
   \bigg| \int_{\bbR^{ N_{i_1}+N_m}}
    \hskip -0.1em
     \frac{ f(y_1,\ldots,y_{i_1}) g(x_1,\ldots,x_m)  dy_1\ldots dy_{i_1} dx_1\ldots dx_m}{(\sum_{i=1}^{i_1}
     |x_i-y_i| + \sum_{i={i_1+1}}^m |x_i|)^{\lambda}}\bigg|
     \le
     C_{\vec p, \vec q}\|f\|_{L^{\vec {\tilde p}}}\|g\|_{L^{\vec q\ \!\!'}},
\]
which is equivalent to
\[
    \Big \| \int_{\bbR^{N_m}}
     \frac{  g(x_1,\ldots,x_m)  dx_1\ldots dx_m}{(\sum_{i=1}^{i_1}
     |x_i-y_i| + \sum_{i={i_1+1}}^m |x_i|)^{\lambda}}
     \Big\|_{L^{\vec {\tilde p}\ \!\!'}}
     \le
     C_{\vec p, \vec q}\|g\|_{L^{\vec q\ \!\!'}}.
\]
It follows that for any $(y_{i_1+1},\ldots,y_m)
\in\bbR^{n_{i_1+1}}\times \ldots \times \bbR^{n_m}$,
\begin{align*}
 &   \Big \| \int_{\bbR^{N_m}}
     \frac{  g(x_1,\ldots,x_m)  dx_1\ldots dx_m}{(\sum_{i=1}^m
     |x_i-y_i|  )^{\lambda}}
     \Big \|_{L_{(y_1,\ldots,y_{i_1})}^{(p'_1,\ldots,p'_{i_1}) }}  \\
 &=    \Big \| \int_{\bbR^{N_m}}
     \frac{  g(x_1,\ldots,x_{i_1},x_{i_1+1}+y_{i_1+1},
     \ldots,x_m+y_m)  dx_1\ldots dx_m}{(\sum_{i=1}^{i_1}
     |x_i-y_i| + \sum_{i={i_1+1}}^m |x_i |)^{\lambda}}
      \Big \|_{L_{(y_1,\ldots,y_{i_1})}^{(p'_1,\ldots,p'_{i_1}) }} \\
&\le
     C_{\vec p, \vec q}
     \| g(x_1,\ldots,x_{i_1},x_{i_1+1}+y_{i_1+1},
     \ldots,x_m+y_m)\|_{L_x^{\vec q\ \!\!'}}\\
&= C_{\vec p, \vec q}
      \|g\|_{L^{\vec q\ \!\!'}}     .
\end{align*}
Since $p'_i=\infty$ for $i_1+1\le i\le m$, we can rewrite the above inequality as
\begin{align*}
 &    \Big \| \int_{\bbR^{N_m}}
     \frac{  g(x_1,\ldots,x_m)  dx_1\ldots dx_m}{(\sum_{i=1}^m
     |x_i-y_i|  )^{\lambda}}
      \Big \|_{L^{\vec p\ \!\!'}}
     \le
C_{\vec p, \vec q}  \|g\|_{L^{\vec q\ \!\!'}}.
\end{align*}
Hence
\begin{align*}
 &    \Big \| \int_{\bbR^{N_m}}
     \frac{ f(y_1,\ldots,y_m)  dy_1\ldots dy_m}{(\sum_{i=1}^m
     |x_i-y_i|  )^{\lambda}}
      \Big \|_{L^{\vec q}}
\le
C_{\vec p, \vec q}  \|f\|_{L^{\vec p}}.
\end{align*}
This completes the proof.
\end{proof}

Next we consider the case  $p_i=q_i$ for some $1\le i\le m$.

\begin{Lemma}\label{Lm:structure}
Let $\vec p = (p_1,\ldots,p_m)$
and $\vec q=(q_1,\ldots,q_m)$ with
$1\le p_i, q_i\le \infty$ %for  $1\le i\le m$
satisfy  (\ref{eq:homo:m}), where $1\le i\le m$.
Suppose that $I_{\lambda}$ is bounded from
$L^{\vec p}$ to $L^{\vec q}$.
If $p_{i_0}=q_{i_0}$ for some $1\le i_0\le   m$,
then the mapping
\[
  f\mapsto \int_{\bbR^{N_m-n_{i_0}}}\!
  \frac{f(y_1,\ldots,y_{i_0-1},y_{i_0+1},\ldots,y_m)
  dy_1\ldots dy_{i_0-1}dy_{i_0+1}\ldots dy_m}
  {(\sum_{i\ne i_0} |x_i - y_i|)^{\lambda-n_{i_0}  }}
\]
is bounded from $L^{(p_1,\ldots,p_{i_0-1},p_{i_0+1},\ldots,p_m)}$
to $L^{(q_1,\ldots,q_{i_0-1},q_{i_0+1},\ldots,q_m)}$.
\end{Lemma}

\begin{proof}
Denote
\begin{align*}
\vec{\tilde p}&=(p_1,\ldots,p_{i_0-1},  p_{i_0+1},  \ldots,p_m), \\
\vec{\tilde q}&=(q_1,\ldots,q_{i_0-1},  q_{i_0+1},  \ldots,q_m),\\
\tilde x &=(x_1,\ldots,x_{i_0-1},  x_{i_0+1},  \ldots,x_m),\\
\tilde y &=(y_1,\ldots,y_{i_0-1},  y_{i_0+1},  \ldots,y_m).
\end{align*}
Fix some $f_1\in L^{\vec{\tilde p}}$.
For any  $a>0$ and $g_1\in L^{\vec {\tilde q}\ \!\!'}$, set
$f_2 =\chi^{}_{\{|y_{i_0}|\le 2a\}}$,
$g_2 =\chi^{}_{\{|x_{i_0}|\le 2a\}}$,
$f(y_1,\ldots,y_m) =   f_1(\tilde y) f_2(y_{i_0})  $
and $g(x_1,\ldots,x_m) =   g_1(\tilde x) g_2(x_{i_0}) $.
Then $f\in L^{\vec p}$ and $g\in L^{\vec q\ \!\!'}$.
We have
\begin{align}
  \int_{\bbR^{ N_m}}
I_{\lambda}|f|(x_1,\ldots,x_m)|g(x_1,\ldots,x_m)|
dx_1 \ldots dx_m
& \le \|I_{\lambda}\|\cdot
 \|f\|_{L^{\vec p}} \|g\|_{L^{\vec q\ \!\!'}} \nonumber  \\
    & \hskip -10mm\approx a^{n_{i_0}}\|I_{\lambda}\|\cdot
\|f_1\|_{L^{\vec {\tilde p}}}\|g_1\|_{L^{\vec {\tilde q}\ \!\!'}}.
  \label{eq:e3}
\end{align}
On the other hand,
\begin{align}
&
\int_{\bbR^{ N_m}}
I_{\lambda}|f|(x_1,\ldots,x_m)|g(x_1,\ldots,x_m)|
dx_1\ldots dx_m  \nonumber  \\
 &=\int_{\bbR^{2N_m}}
   \frac{|f(y_1,\ldots,y_m)g(x_1,\ldots,x_m)|
     dx_1dy_1\ldots dx_mdy_m}
      {(\sum_{i=1}^m |x_i-y_i|)^{\lambda}} \nonumber \\
&\ge
\int_{|\tilde x-\tilde y|\le a }
  |f_1(\tilde y)g_1(\tilde x)|
     d\tilde x d\tilde y
        \int_{|x_{i_0}-y_{i_0}|\le |\tilde x- \tilde y|}
 \frac{|f_2(y_{i_0})g_2(x_{i_0})| dx_{i_0} dy_{i_0}}
     {(\sum_{i=1}^m |x_i-y_i|)^{\lambda}}
 \nonumber \\
&\gtrsim  \int_{|\tilde x - \tilde y|\le  a }
  \frac{|f_1(\tilde y)g_1(\tilde x)|
     d\tilde x d\tilde y}
     {|\tilde x-\tilde y|^{\lambda}}\cdot
     |\tilde x - \tilde y|^{n_{i_0}}
  \Big(2a- |\tilde x-\tilde y|\Big)^{n_{i_0}}
  \nonumber
   \\
&\gtrsim  a^{n_{i_0}}  \int_{|\tilde x - \tilde y|\le a }
  \frac{|f_1(\tilde y)g_1(\tilde x)|
     d\tilde x d\tilde y}
     {|\tilde x-\tilde y|^{\lambda-n_{i_0}}}.\label{eq:e4}
\end{align}
Denote
\[
  L_a f_1(\tilde x)
  = \int_{|\tilde x - \tilde y|\le a }
  \frac{|f_1(\tilde y) |
       d\tilde y}
     {|\tilde x-\tilde y|^{\lambda-n_{i_0}}}.
\]
We see from (\ref{eq:e3}) and (\ref{eq:e4}) that for any
$g_1\in L^{\vec {\tilde q}\ \!\!'}$,
$
  |\langle L_a f_1, g_1\rangle |
  \lesssim  \|f_1\|_{L^{\vec{\tilde p}}}\|g_1\|_{L^{\vec {\tilde q}\ \!\!'}}.
$
%It follows from the Banach-Steinhaus theorem that
Hence
\[
    \| L_a f_1\|_{L^{\vec {\tilde q}}} \lesssim  \|f_1\|_{L^{\vec{\tilde p}}},\qquad a>0.
\]
Letting $a\rightarrow \infty$, we see from the monotone convergence theorem that
\[
  \left\| \int_{\bbR^{N_m-n_{i_0}}} \frac{|f_1(\tilde y)| d\tilde y}
  {|\tilde x-\tilde y|^{\lambda-n_{i_0}}}
   \right\|_{L^{\vec {\tilde q}}} \lesssim
   \|f_1\|_{L^{\vec {\tilde p}}}.
\]
This completes the proof.
\end{proof}

The following lemma focuses on the case $(p_i,q_i)=(1,\infty)$
for some $1\le i\le m$.

\begin{Lemma}\label{Lm:1:infty}
Suppose that $\vec p=(p_1,\ldots,p_m)$
and $\vec q = (q_1,\ldots,q_m)$
with $1\le p_i, q_i\le \infty$ for all $1\le i\le m$ and $(p_{i_0},q_{i_0})
=(1,\infty)$ for some $1\le i_0\le m$.
Then
$I_{\lambda}$ is bounded from $L^{\vec p}$ to $L^{\vec q}$
if and only if
\begin{equation}\label{eq:ea:a20}
 T_{\lambda}:\,  f\mapsto \int_{\bbR^{N_m-n_{i_0}}}\!
  \frac{f(y_1,\ldots,y_{i_0-1},y_{i_0+1},\ldots,y_m)
  dy_1\ldots dy_{i_0-1}dy_{i_0+1}\ldots dy_m}
  {(\sum_{i\ne i_0} |x_i - y_i|)^{\lambda }}
\end{equation}
is bounded from $L^{(p_1,\ldots,p_{i_0-1},p_{i_0+1},\ldots,p_m)}$
to $L^{(q_1,\ldots,q_{i_0-1},q_{i_0+1},\ldots,q_m)}$.
\end{Lemma}

\begin{proof}
Let $\vec{\tilde p} $, $\vec{\tilde q} $, $\tilde x$ and $\tilde y$ be defined as in the proof of Lemma~\ref{Lm:structure}.

\textit{Sufficiency}.
Assume that $T_{\lambda}$ is bounded from
$L^{\vec{\tilde p}}$
to
$L^{\vec{\tilde q}}$.

For any $f\in L^{\vec p}$, since $q_{i_0}=\infty$, we have
\begin{align}
\|I_{\lambda}f\|_{L_{(x_1,\ldots,x_m)}^{(q_1,\ldots,q_m)}}
& \le
\|I_{\lambda}f\|_{L_{(x_{i_0},x_1,\ldots,x_{i_0-1},x_{i_0+1},
\ldots,x_m)}^{(q_{i_0},q_1,\ldots,q_{i_0-1},q_{i_0+1},
\ldots,q_m)}}.
 \label{eq:ea:a15}
\end{align}
By Young's inequality, we get
\begin{align*}
\|I_{\lambda}f\|_{L_{x_{i_0}}^{q_{i_0}}}
&\lesssim
   \int_{\bbR^{N_m-n_{i_0}}}
  \frac{\|f(y_1,\ldots,y_m)\|_{L_{y_{i_0}}^{p_{i_0}}}
   dy_1\ldots dy_{i_0-1}dy_{i_0+1}\ldots dy_m}
   {(\sum_{\substack{1\le i\le m,
      i\ne i_0}} |x_i-y_i|)^{\lambda}}.
\end{align*}
Applying the boundedness of  $T_{\lambda}$  from
$L^{\vec {\tilde p}}$
to
$L^{\vec {\tilde q}}$,
we have
\begin{align*}
\Big\|\|I_{\lambda}f\|_{L_{x_{i_0}}^{q_{i_0}}}
\Big\|_{L_{\tilde x}^{\vec{\tilde q}}}
 \lesssim
\Big\| \|f(y_1,\ldots,y_m)\|_{L_{y_{i_0}}^{p_{i_0}}}
\Big\|_{L_{\tilde y}^{\vec{\tilde p}}} .
\end{align*}
It follows from (\ref{eq:ea:a15}) that
\[
  \|I_{\lambda}f\|_{L^{\vec q}}
  \lesssim
\Big\| \|f(y_1,\ldots,y_m)\|_{L_{y_{i_0}}^{p_{i_0}}}
\Big\|_{L_{\tilde y}^{\vec{\tilde p}}} .
\]
Since $p_{i_0}=1\le p_i$ for $1\le i\le i_0-1$,
applying Minkowski's inequality, we get
$\|I_{\lambda f}\|_{L^{\vec q}}\lesssim \|f\|_{L^{\vec p}}$.

\textit{Necessity}.
Assume that $I_{\lambda}$ is bounded from $L^{\vec p}$ to $L^{\vec q}$.
For any $\tilde f\in L^{\vec{\tilde p}}$ and $\delta>0$,
set
\[
  f_{\delta}(y_1,\ldots,y_m) =  \frac{1}{\delta^{n_{i_0}}}\chi^{}_{\{|y_{i_0}|\le \delta\}} (y_{i_0})\tilde f(y_1,\ldots,y_{i_0-1},y_{i_0+1},
\ldots,y_m).
\]
By the assumption,
\begin{align*}
\left\| \int_{\bbR^{N_m}}
   \frac{|f_{\delta}(y_1,\ldots,y_m) |dy_1\ldots dy_m}{(\sum_{i=1}^m |x_i-y_i|)^{\lambda}}
 \right\|_{L_{(x_1,\ldots,x_m)}^{(q_1,\ldots,q_m)}}
\lesssim \|f_{\delta}\|_{L^{\vec p}}.
\end{align*}
By the mean-value theorem, there is some
$\xi_{\delta}\in\bbR^{n_{i_0}}$ with $|\xi_{\delta}|\le \delta$ such that
\begin{align*}
\bigg\|\int_{\bbR^{N_m-n_{i_0}}} \hskip -12.5pt
   \frac{|\tilde f(y_1,\ldots,y_{i_0-1},y_{i_0+1},
\ldots,y_m)|
 dy_1\ldots dy_{i_0-1}dy_{i_0+1}\ldots dy_m}
  {(\sum_{\substack{1\le i\le m \\ i\ne i_0}} |x_i-y_i| + |x_{i_0}-\xi_{\delta}|)^{\lambda}}
 \bigg\|_{L_{x}^{\vec q}}
% \hskip -12.5pt
\lesssim \|f_{\delta}\|_{L^{\vec p}}.
\end{align*}
Letting $\delta\rightarrow 0$, we see from Fatou's lemma
that for any $\tilde f\in L^{\vec {\tilde p}}$,
\begin{align*}
\bigg\| \int_{\bbR^{N_m-n_{i_0}}} \hskip -12.5pt
   \frac{|\tilde f(y_1,\ldots,y_{i_0-1}, y_{i_0+1},\ldots,y_m)|
   dy_1\ldots dy_{i_0-1}dy_{i_0+1} \ldots dy_m}
   {(\sum_{\substack{1\le i\le m \\
        i\ne i_0}} |x_i-y_i| + |x_{i_0}|)^{\lambda}}
\bigg\|_{L_{x}^{\vec q}}
 %\hskip -12.5pt
\lesssim \|\tilde f\|_{L^{\vec {\tilde p}}}.
\end{align*}
By the duality, we get
\begin{align*}
\bigg\| \int_{\bbR^{N_m}}
   \frac{  g(x_1,\ldots,x_m) dx_1 \ldots dx_m}
   {(\sum_{\substack{1\le i\le m \\
        i\ne i_0}} |x_i-y_i| + |x_{i_0}|)^{\lambda}}
 \bigg\|_{L_{\tilde y
  }^{\vec{\tilde p}\ \!\!'}}
\lesssim \|g\|_{L^{\vec q\ \!\!'}},\,\, \forall g
\in L^{\vec q\ \!\!'}.
\end{align*}
Since $q'_{i_0}=1$, by setting
\[
  g_{\delta}(x_1,\ldots,x_m) =  \frac{1}{\delta^{n_{i_0}}}\chi^{}_{\{|x_{i_0}|\le \delta\}} (x_{i_0})\tilde g(x_1,\ldots,x_{i_0-1},x_{i_0+1},
\ldots,x_m)
\]
with $\tilde g
\in L^{\vec {\tilde q}\ \!\!'}$
and letting $\delta\rightarrow 0$, we get
%for any $\tilde g \in L^{\vec {\tilde q}\ \!\!'}$,
\begin{align*}
&\bigg\| \int_{\bbR^{N_m-n_m}}\hskip -4mm
   \frac{ \tilde g(x_1,\ldots,x_{i_0-1},x_{i_0+1},
\ldots,x_m)  dx_1 \ldots dx_{i_0-1}dx_{i_0+1}
\ldots dx_m}
   {(\sum_{\substack{1\le i\le m \\
        i\ne i_0}} |x_i-y_i|  )^{\lambda}}
 \bigg\|_{L_{\tilde y
  }^{\vec{\tilde p}\ \!\!'}}
\lesssim \|\tilde g\|_{L^{\vec {\tilde q}\ \!\!'}}.
\end{align*}
Hence $T_{\lambda}$ is bounded from  $L^{\vec {\tilde q}\ \!\!'}$
to $L^{\vec {\tilde p}\ \!\!'}$.
Applying the duality again, we get that
$T_{\lambda}$ is bounded from  $L^{\vec {\tilde p}}$
to $L^{\vec {\tilde q} }$.
\end{proof}

The following lemma is used in the proof of the necessity.
\begin{Lemma}\label{Lm:pi}
Let $\vec p$, $\vec q$ and $\lambda$
be such that
(\ref{eq:homo:m}) is true, $1\le p_i\le q_i\le \infty$ for
all $1\le i\le m$ and there is some $i_0$ such that
$1<p_{i_0}<\infty$.
Suppose that $q_m<\infty$ and there is some $1< v\le m$ such that
$p_{v-1}>1$ and $p_i=1$ for all $v\le i\le m$.
If $I_{\lambda}$ is bounded  from $L^{\vec p}$
to $L^{\vec q}$, then we have
\[
  q_m \ge p_{v-1}.
\]
\end{Lemma}

\begin{proof}
We prove the conclusion in two steps.

(S1)\,\, We show that $q_m \ge p_{v-1}$ when $p_{v-1}<\infty$.

By Lemma~\ref{Lm:end points}, for any $f\in
L^{(p_1,\ldots,p_{v-1})}$, we have
\[
  \bigg\| \int_{\bbR^{n_1+\ldots+n_{v-1}}}
     \frac{f(y_1,\ldots,y_{v-1}) dy_1\ldots dy_{v-1}}{(\sum_{i=1}^{v-1}
     |x_i-y_i| + \sum_{i=v}^m |x_i|)^{\lambda}}
     \bigg\|_{L^{\vec q}}
     \lesssim \|f\|_{L^{(p_1,\ldots,p_{v-1})}}.
\]
Let
\[
  f(y_1,\ldots,y_{v-1}) = \frac{\chi^{}_{\{|y_{v-1}|< 1/2\}}(y_{v-1})}
  {(\sum_{i=1}^{v-1} |y_i|)^{n_1/p_1+\ldots+n_{v-1}/p_{v-1}}
  (\log 1/|y_{v-1}|)^{(1+\varepsilon)/p_{v-1}}},
\]
where $\varepsilon>0$ is a constant to be determined later.
It is easy to check that $f\in L^{(p_1,\ldots,p_{v-1})}$.
When $|x_i|\le 1/2$, $1\le i\le m$,
\begin{align}
Tf(x_1,\ldots,x_m)&:=\int_{\bbR^{n_1+\ldots+n_{v-1}}}
     \frac{f(y_1,\ldots,y_{v-1}) dy_1\ldots dy_{v-1}}{(\sum_{i=1}^{v-1}
     |x_i-y_i| +  \sum_{i=v}^m |x_i|)^{\lambda}} \nonumber \\
&\ge
\int_{\substack{ |x_i|^2\le |y_i|\le |x_i|\\ 1\le i\le v-1}}
    \frac{f(y_1,\ldots,y_{v-1}) dy_1\ldots dy_{v-1}}{(\sum_{i=1}^{v-1}
     |x_i-y_i| +  \sum_{i=v}^m |x_i|)^{\lambda}} \nonumber \\
&\gtrsim \frac{|x_1|^{n_1}\cdots |x_{v-1}|^{n_{v-1}}}{(\sum_{i=1}^m|x_i|)^{\lambda+n_1/p_1+\ldots
+n_{v-1}/p_{v-1}}(\log 1/|x_{v-1}|)^{(1+\varepsilon)/p_{v-1}}}.
  \label{eq:ea5a}
\end{align}
Hence
\begin{align*}
\|Tf\|_{L^{\vec q}}
&\gtrsim
\|Tf\cdot \chi^{}_{\{ |x_m|^2< |x_i|< |x_m|<1/2,
1\le i\le m-1\}}\|_{L^{\vec q}} \\
& \gtrsim
   \Big \|\frac{\chi^{}_{\{|x_m|<1/2\}}(x_m)}
   {|x_m|^{n_m/q_m}(\log 1/|x_m|)^{(1+\varepsilon)/p_{v-1}}} \Big\|_{L_{x_m}^{q_m}}.
 \end{align*}
If $q_m<p_{v-1}$, we can choose some $\varepsilon>0$ such that $(1+\varepsilon)q_m/p_{v-1}<1$. Consequently,
$\|Tf\|_{L^{\vec q}}=\infty$, which is a contradiction.
Hence $q_m\ge  p_{v-1}$.

(S2)\,\, We show that $p_{v-1}<\infty$.

We claim  that $q_m>1$.
Recall that $1<p_{i_0}<\infty$.
Let $i_1 = \max\{i:\, 1<p_i<\infty\}$.
Then $i_0\le i_1 \le m-1$ and for any $i_1<i\le m$,
either $p_i=1$ or $p_i=\infty$.
Since $q_i\ge p_i$, we have $q_i=p_i$ when $p_i=\infty$.
Denote
$\{1,\ldots,i_1\}\cup \{i:\, i_1\le i\le m \mbox{ and } p_i<\infty\}$
by $\{j_l:\, 1\le l \le r\}$, where $j_1<\ldots<j_r$.
Then $p_{j_l} = p_l$ for all $1\le l\le i_1$.
By Lemma~\ref{Lm:structure},
 the mapping
\[
   f\mapsto \int_{\bbR^{n_{j_1}+\dots+n_{j_r}}}
     \frac{ f(y_{j_1},\ldots,y_{j_r})
        dy_{j_1}\ldots dy_{j_r}}
        {(\sum_{l=1}^r |x_{j_l}-y_{j_l}|)^{
         \sum_{l=1}^r (n_{j_l}/p'_{j_l} + n_{j_l}/q_{j_l})}
        }
\]
is bounded from $L^{(p_{j_1},\ldots,p_{j_r})}$
to $L^{(q_{j_1},\ldots,q_{j_r})}$.
Since $1<p_{j_{i_1}}<\infty$ and $p_{j_l}=1$ for all $i_1+1\le l\le r$, we see from (S1) that $q_m = q_{j_r}\ge p_{j_{i_1}} >1$. Hence $1<q_m<\infty$.

By Lemma~\ref{Lm:end points}, for any $f\in L^{(p_1,\ldots,p_{v-1})}$,
\begin{align*}
 \bigg \|\int_{\bbR^{n_1+\ldots+n_{v-1}}}
     \frac{f(y_1,\ldots,y_{v-1}) dy_1\ldots dy_{v-1}}{(\sum_{i=1}^{v-1}
     |x_i-y_i| + \sum_{i={v}}^m |x_i|)^{\lambda}}
     \bigg\|_{L^{\vec q}}
     \lesssim \|f\|_{L^{(p_1,\ldots,p_{v-1})}}.
\end{align*}
It follows from the duality that
for any $g\in L^{\vec q\ \!\!'}$,
\begin{align}
 \bigg\|\int_{\bbR^{N_m}}
     \frac{g(x_1,\ldots,x_m) dx_1\ldots dx_m}{(\sum_{i=1}^{v-1}
     |x_i-y_i| + \sum_{i={v}}^m |x_i|)^{\lambda}}
     \bigg\|_{L_{(y_1,\ldots,y_{v-1})}^{(p'_1,\ldots,p'_{v-1})}}
     \lesssim \|g\|_{L^{\vec q\ \!\!'}}. \label{eq:s:e2a}
\end{align}
Let
\[
  g(x_1,\ldots,x_m) = \frac{\chi^{}_{\{|x_m|< 1/2\}}(x_m)}
  {(\sum_{i=1}^m |x_i|)^{n_1/q'_1+\ldots+n_m/q'_m}
  (\log 1/|x_m|)^{(1+\varepsilon)/q'_m}},
\]
where $\varepsilon>0$ is a constant.
It is easy to check that  $g\in L^{\vec q\ \!\!'}$.

For any $y:=(y_1,\ldots,y_{v-1})\in E:= \{(y_1,\ldots,y_{v-1}):\, |y_{v-1}|^2\le
|y_i|\le |y_{v-1}|, 1\le i\le v-2
\mbox{ and } |y_{v-1}|<1/2
\} $, we have
\begin{align*}
I_g(y)
&:=
\int_{\bbR^{N_m}}
     \frac{g(x_1,\ldots,x_m) dx_1\ldots dx_m}{(\sum_{i=1}^{v-1}
     |x_i-y_i| + \sum_{i={v}}^m |x_i|)^{\lambda}} \\
&\ge
\int_{\substack{|y_i|^2 < |x_i|<|y_i|, 1\le i\le v-1 \\
    |y_{v-1}|^2<|x_i|<|y_{v-1}|, v\le i\le m}}
     \frac{g(x_1,\ldots,x_m) dx_1\ldots dx_m}{(\sum_{i=1}^{v-1}
     |x_i-y_i| + \sum_{i={v}}^m |x_i|)^{\lambda}} \\
&\gtrsim
   \frac{ |y_{v-1}|^{n_{v}+\ldots+n_m} \prod_{i=1}^{v-1} |y_i|^{n_i} }
   {(\sum_{i=1}^{v-1} |y_i|)^{\lambda + \sum_{i=1}^m n_i/q'_i}
     (\log 1/|y_{v-1}|)^{(1+\varepsilon)/q'_m} }.
\end{align*}
Hence for $(y_2,\ldots,y_{v-1})\in \{(y_2,\ldots,y_{v-1}):\, |y_{v-1}|^2\le
|y_i|\le |y_{v-1}|, 2\le i\le v-2$
  and $ |y_{v-1}|<1/2
\} $,
\begin{align*}
 \| I_g(y) \|_{L_{y_1}^{p'_1}}
 &\gtrsim
 \bigg\|
   \frac{ |y_{v-1}|^{n_{v}+\ldots+n_m} \prod_{i=1}^{v-1} |y_i|^{n_i}
    \chi^{}_{\{ |y_{v-1}|^2 < |y_1| < |y_{v-1}| \}}(y_1)}
   {(\sum_{i=1}^{v-1} |y_i|)^{\lambda + \sum_{i=1}^m n_i/q'_i}
     (\log 1/|y_{v-1}|)^{(1+\varepsilon)/q'_m} }
 \bigg\|_{L_{y_1}^{p'_1}} \\
 &\gtrsim
   \frac{ |y_{v-1}|^{n_{v}+\ldots+n_m+n_1+n_1/p'_1} \prod_{i=2}^{v-1} |y_i|^{n_i} }
   {(\sum_{i=2}^{v-1} |y_i|)^{\lambda + \sum_{i=1}^m n_i/q'_i}
     (\log 1/|y_{v-1}|)^{(1+\varepsilon)/q'_m} }.
\end{align*}
Computing the $L_{y_j}^{p'_j}$ norm with respect to
$y_j\in \{y_j:\,|y_{v-1}|^2 \le |y_j| \le |y_{v-1}|\}$ successively for $2\le j\le v-2$, we have
\begin{align*}
  \| I_g(y)\chi^{}_{E}(y)\|_{L_{(y_1,\ldots,y_{v-2})
    }^{(p'_1,\ldots,p'_{v-2})}}
&\gtrsim
   \frac{ |y_{v-1}|^{N_m  + n_1/p'_1+\ldots +n_{v-2}/p'_{v-2}}
     \chi^{}_{\{|y_{v-1}|\le 1/2\}}(y_{v-1})}
   {   |y_{v-1}|^{\lambda + \sum_{i=1}^m n_i/q'_i}
     (\log 1/|y_{v-1}|)^{(1+\varepsilon)/q'_m} } \\
&=
   \frac{ \chi^{}_{\{|y_{v-1}|\le 1/2\}}(y_{v-1})  }
   {   |y_{v-1}|^{ n_{v-1}/p'_{v-1}}
     (\log 1/|y_{v-1}|)^{(1+\varepsilon)/q'_m} }.
\end{align*}

If  $p_{v-1}=\infty$, then $p'_{v-1}=1<q'_m$. We can choose $\varepsilon>0$ small enough such that
$p'_{v-1} (1+\varepsilon)/q'_m<1$. Consequently,
\[
  \| I_g(y)\chi^{}_{E}(y)\|_{L_{(y_1,\ldots,y_{v-1})}^{(p'_1,\ldots,p'_{v-1})}}
  =\infty,
\]
which contradicts (\ref{eq:s:e2a}). Hence $p_{v-1}<\infty$.
\end{proof}

The Calder\'on-Zygmund theory for Bochner spaces
were well established, e.g., see
\cite[Theorem II.1.3]{Rubio1986},
\cite[Theorem 5.17]{Duoandikoetxea2001},
\cite[Theorem 5.6.1]{Grafakos2014}
 or
\cite[Theorem 2.1]{Torres2015}.
Here we present a similar result
%on the singular integral operator
for mixed-norm Lebesgue spaces.

\begin{Proposition}\label{prop:singular T}
Let $n_1$, $\ldots$, $n_m$ be positive integers, $N=n_1+\ldots+n_m$,
$\vec u=(u_1,\ldots,u_m)$ and
$\vec v=(v_1,\ldots,v_m)$ with $1\le u_i,v_i\le\infty$.
Suppose that $K(x,y; \bar x,\bar y)$ is a measurable function
defined on $\bbR^{N+n}\times\bbR^{N+n}\setminus
\{(x,y; \bar x,\bar y):\, |x-\bar x| + |y-\bar y|=0\}$ such that
for some $A >0$ and
$\alpha >\max \{\sum_{i=1}^m  n_i/u'_i,
\sum_{i=1}^m  n_i/v_i \}$,
\begin{equation}\label{eq:ec:K}
  |K(x,y; \bar x,\bar y)| \le \frac{A}{(|x-\bar x| + |y-\bar y|)^{\alpha}}.
\end{equation}
Let $T$ be a linear operator which is bounded from
$L^r(L^{\vec u})$ to $L^r(L^{\vec v})$ for some
$1<r<\infty$
and satisfies the followings,

\begin{enumerate}
\item If  $\|f(\cdot,\bar y)\|_{L^{\vec u}}\in L^{\infty}$ is compactly supported,
then
\[
  Tf(x,y)= \int_{\bbR^{N+n}} K(x,y; \bar x,\bar y)f(\bar x, \bar y)
  d\bar x d\bar y,\qquad y\not\in\mathrm{supp}\,   \|f(\cdot,\bar y)\|_{L^{\vec u}}.
\]

\item For any $y\ne \bar y$,
\[
  T_{y,\bar y} h(x):= \int_{\bbR^N} K(x,y; \bar x,\bar y)h(\bar x)
  d\bar x
\]
defines a bounded linear operator from $L^{\vec u}$ to $L^{\vec v}$
and there is some constant $C$ such
that
\begin{align*}
\int_{|y-\bar y|\ge 2|\bar y-\bar z|}
\| T_{y,\bar y} - T_{y,\bar z}\|_{L^{\vec u}\rightarrow L^{\vec v}}   dy
 \le  C, \qquad \forall \bar y\ne \bar z,\\
\int_{|y-\bar y|\ge 2|  y-  z|}
\| T_{y,\bar y} - T_{z,\bar y}\|_{L^{\vec u}\rightarrow L^{\vec v}}   d\bar y
 \le  C,\qquad \forall   y\ne z .
\end{align*}
\end{enumerate}
Then $T$ extends to a bounded linear operator from $L^p(L^{\vec u})$
to $L^p(L^{\vec v})$ for all $1<p<\infty$.
\end{Proposition}

\textbf{Remark}.\,\,
Note that the condition (\ref{eq:ec:K}) ensures that $T_{y,\bar y} h$ is well defined pointwise on $\bbR^N$.
Moreover, $K$ meets (\ref{eq:ec:K}) whenever
\[
  |K(x,y; \bar x,\bar y)| \le \frac{A}{(|x-\bar x| + |y-\bar y|)^{N+n}}.
\]

\begin{proof}
Fix some $f\in L^1(L^{\vec u})\bigcap L^r(L^{\vec u})$ and $\lambda>0$.
Applying the Calder\'{o}n-Zygmund decomposition  of  $\|f(\cdot,y)\|_{L^{\vec u}}$
at the height $\lambda$, we get disjoint cubes $\{Q_{k}:\, k\in \mathbb{K}\}$  such that
\[
    \|f(\cdot,y)\|_{L^{\vec u}} \le \lambda,\qquad y\not\in \bigcup_{k\in\mathbb K} Q_k
\]
and
\[
   \lambda \le \frac{1}{|Q_k|} \int_{Q_k} \|f(\cdot,y)\|_{L^{\vec u}} dy \le 2^n\lambda, \qquad k\in \mathbb{K}.
\]
Define
\[
g(x,y)=\left\{ \begin{array}{ll}
  f(x,y),  &y\not\in \bigcup_{k\in\mathbb K} Q_k, \\
  |Q_{k}|^{-1}\int_{Q_{k}}f(x,z)dz,  &y\in  Q_{k},
\end{array}
\right.
\]
and
\[
b_{k}(x,y)=\left\{\begin{array}{ll}
  0, & y\not\in Q_{k} ,\\
  f(x,y)-|Q_{k}|^{-1}\int_{Q_{k}}f(x,z)dz, &y\in Q_{k}.
\end{array}\right.
\]
Let   $b(x,y) = \sum_{k\in\mathbb K} b_k(x,y)$.
Then $f=g+b$ and  for any $\lambda>0$,
\begin{equation}\label{eq:lambda}
  \{y:\,\|(Tf)(\cdot,y)\|_{L^{\vec u}} \ge \lambda\}
  \subset  \{y:\,\|(Tg)(\cdot,y)\|_{L^{\vec u}} \ge \frac{\lambda}{2}\} \cup
   \{y:\,\|(Tb)(\cdot,y)\|_{L^{\vec u}} \ge \frac{\lambda}{2}\}.
\end{equation}

Fix some $\gamma>2n^{1/2}+1$.
For each $k\in\mathbb K$,  Let $Q_k^*$
be the cube with the same center $y_k$ as  $Q_{k}$
and $\gamma$ times the length. Let $B^{\ast}=\bigcup_{k\in\mathbb K}
Q_{k}^{\ast}$,
$G^{\ast}=\mathbb{R}^{n}\setminus{B}^{\ast}$ and
$G = \bbR^n\setminus \bigcup_{k\in\mathbb K}
Q_{k}$.
Then we have
\begin{align}
\hskip 10mm&\hskip -10mm\Big|\{y:\,\|Tg(\cdot,y)\|_{L^{\vec v}}\geq \frac{\lambda}{2}|\}\Big| \nonumber \\
&\le \frac{2^r}{\lambda^r} \int_{\{y:\,\|Tg(\cdot,y)\|_{L^{\vec v}} \geq \frac{\lambda}{2}|\}} \|Tg(\cdot,y)\|_{L^{\vec v}}^r dy \nonumber \\
&= \frac{2^r\|T\|_{L^r(L^{\vec u})\rightarrow L^r(L^{\vec v})}^r}{\lambda^r}
      \int_{\bbR^{ n}} \|g(\cdot,y)\|_{L^{\vec u}}^r dy \nonumber \\
&\le \frac{C}{\lambda^r}\Big(
   \int_{G} \|g(\cdot,y)\|_{L^{\vec u}}^rdy + \sum_{k\in\mathbb{K}} \int_{Q_k}\!\!  \|g(\cdot,y)\|_{L^{\vec u}}^r dy
       \Big) \nonumber \\
&\le \frac{C}{\lambda^r}\Big(
   \int_{G}\! \lambda^{r-1} \|f(\cdot,y)\|_{L^{\vec u}} dy +(2^n\lambda)^{r-1}  \sum_{k\in\mathbb{K}} \int_{Q_k}\!\!\! \|f(\cdot,y)\|_{L^{\vec u}}dy
       \!\Big) \nonumber \\
&\lesssim     \frac{1}{\lambda} \|f\|_{L^1(L^{\vec u})}.
    \label{eq:tf1}
\end{align}
On the other hand, we have
\begin{align}
\Big|\{y:\, \|Tb(\cdot,y)\|_{L^{\vec v}}\geq \frac{\lambda}{2} \}\Big|
&\le  |B^*| + \Big|\{y\in G^*:\,\|Tb(\cdot,y)\|_{L^{\vec v}}\geq \frac{\lambda}{2} \}\Big| \nonumber \\
&\le  \gamma^n \Big| \bigcup_{k\in\mathbb K} Q_k\Big| + \frac{2}{\lambda}\int_{G^*}\|Tb(\cdot,y)\|_{L^{\vec v}} dy\nonumber \\
&\le   \frac{\gamma^n}{\lambda}\|f\|_{L^1(L^{\vec u})} +
  \frac{2}{\lambda} \sum_{k\in\mathbb K} \int_{G^*}\|Tb_k(\cdot,y)\|_{L^{\vec v}}dy. \label{eq:K:e2}
\end{align}
Note that
\begin{align*}
\int_{G^*}\|Tb_k(\cdot,y)\|_{L^{\vec v}}dy
&\le   \int_{\bbR^n\setminus Q_k^*}\|Tb_k(\cdot,y)\|_{L^{\vec v}}dy \nonumber \\
&=   \int_{\bbR^n\setminus Q_k^*} dy
  \Big\|\int_{\bbR^N\times Q_k} K(x ,y ;\bar x,\bar y) b_k(\bar x,\bar y) d\bar xd\bar y \Big\|_{L_x^{\vec v}}.
\end{align*}
Since $b_k\in L^1(L^{\vec u})\cap L^r (L^{\vec u})$, by (\ref{eq:ec:K}),
$\int_{\bbR^N\times Q_k} |K(x ,y ;\bar x,\bar y) b_k(\bar x,\bar y)| d\bar xd\bar y<\infty$. Now we see from Fubini's theorem that
\begin{eqnarray}
 \arraycolsep 0pt
 &&\int_{G^*}\|Tb_k(\cdot,y)\|_{L^{\vec v}}dy\nonumber \\
&\le&   \int_{\bbR^n\setminus Q_k^*} dy
  \Big\|\int_{\bbR^N \times Q_k}( K(x ,y ;\bar x,\bar y)-K(x ,y ;\bar x, y_k)) b_k(\bar x,\bar y) d\bar xd\bar y\Big\|_{L_x^{\vec v}} \nonumber \\
&\le &   \int_{\bbR^N\setminus Q_k^*} dy
  \int_{Q_k}\Big\|\int_{\bbR^m }( K(x ,y ;\bar x,\bar y)-K(x ,y ;\bar x, y_k)) b_k(\bar x,\bar y) d\bar x\Big\|_{L_x^{\vec v}}d\bar y \nonumber \\
&\le &   \int_{\bbR^N\setminus Q_k^*} dy
  \int_{Q_k}
  \|T_{y,\bar y} - T_{y,y_k}\|_{L^{\vec u}\rightarrow L^{\vec v}}
     \|b_k(\cdot,\bar y)\|_{L^{\vec u}}
  d\bar y \nonumber \\
 &\le&
 C  \int_{  Q_k}
    \| b_k(\cdot,\bar y)\|_{L^{\vec u}} d\bar y      \nonumber \\
 &\le&
 C  \int_{  Q_k}
    \|f(\cdot,\bar y)\|_{L^{\vec u}} d\bar y.
     \label{eq:tem2}
\end{eqnarray}
It follows from (\ref{eq:K:e2}) that
\[
  \Big|\{y:\, \|Tb(\cdot,y)\|_{L^{\vec v}}\geq \frac{\lambda}{2} \}\Big|
  \le \frac{C}{\lambda} \|f\|_{L^1(L^{\vec u})}.
\]
By (\ref{eq:lambda}) and (\ref{eq:tf1}), we get
for any $f\in L^1(L^{\vec u})\cap L^r(L^{\vec u})$,
\begin{equation}\label{eq:K:e3}
\|Tf\|_{L^{1,\infty}(L^{\vec v})} \le C \|f\|_{L^1(L^{\vec u})}.
\end{equation}

For any $f\in L^1(L^{\vec u})$ and $k\ge 1$,
let $f_k = f \cdot \chi^{}_{\{\|f(\cdot,y)\|_{L^{\vec u}} \le k\}}$.
Then $f_k \in L^1(L^{\vec u})\cap L^r(L^{\vec u})$.
Moreover, we see from the dominated convergence theorem
for $L^1(\bbR^n)$ that
\[
  \lim_{k\rightarrow\infty}\|f-f_k\|_{L^1(L^{\vec u})} = 0.
\]
It follows from (\ref{eq:K:e3}) that
$T$ extends to a linear bounded operator from
$L^1(L^{\vec u})$ to $ L^{1,\infty}(L^{\vec v})$.
Since $T$ is bounded from
$L^r(L^{\vec u})$ to $ L^r(L^{\vec v})$,
by Proposition~\ref{prop:interpolation} and
Corollary~\ref{Co:Lorentz},
$T$ is bounded from
$L^p(L^{\vec u})$ to $ L^p(L^{\vec v})$ whenever $1<p<r$.

On the other hand, we see from \cite[\S2, Theorem 2]{Benedek1961}
that for any measurable function $f$ on $\bbR^{N+n}$,
\[
  \|f\|_{L^r(L^{\vec u})}
     = \sup_{\|g\|_{L^{r'}(L^{\vec u\ \!\!'})}=1}
         \Big|\int fg \Big|.
\]
Hence $T$ is bounded from $L^r(L^{\vec u})$ to
$L^r(L^{\vec v})$ if and only if
$T^*$ is bounded from
$L^{r'}(L^{\vec v\ \!\!'})$
to $L^{r'}(L^{\vec u\ \!\!'})$.
Moreover,  if  $\|g(\cdot,y)\|_{L^{\vec v\ \!\!'}}\in L^{\infty}$
is compactly supported,
then
\[
  T^*g(\bar x, \bar y)= \int_{\bbR^{N+n}} K(x,y; \bar x,\bar y)g( x, y)
  d x dy,\qquad \bar y\not\in\mathrm{supp}\,   \|g(\cdot, y)\|_{L^{\vec u}}.
\]
The previous arguments show that
$T^*$ is bounded from
$L^{p}(L^{\vec v\ \!\!'})$
to $L^{p}(L^{\vec u\ \!\!'})$ whenever $1<p<r'$.
Hence $T$ is bounded from
$L^{p}(L^{\vec u})$
to $L^{p}(L^{\vec v})$ whenever $r<p<\infty$.
\end{proof}

We are now ready to prove Theorem~\ref{thm:mixed fractional}.
We split the proof into two subsections.
One is for the necessity, and the other is for the sufficiency.

\subsection{Proof of Theorem~\ref{thm:mixed fractional}:
The Necessity}

Suppose that $I_{\lambda} $ is bounded from $L^{\vec p}$ to $L^{\vec q}$.
We prove the conclusion in several steps.

Recall that the homogeneity condition (\ref{eq:homo:m})
and the relationship $p_i\le q_i$ for the case
 $1<p_i<\infty$   were proved in
\cite{AdamsBagby1974}. For the sake of completeness, we include a proof here.

(S1)\,   We prove that $\vec p$, $\vec q$ and $\lambda$ meet (\ref{eq:homo:m}).

Take some $f\in L^{\vec p}$.
  Let $a>0$ and $f_a = f(\cdot/a)$. It is easy to check that
\begin{align*}
I_{\lambda} f_a(x) = a^{N_m-\lambda} I_{\lambda} f\Big(\frac{x}{a}\Big).
\end{align*}
By the hypothesis, $\|I_{\lambda} f_a\|_{L^{\vec q}} \lesssim \|f_a\|_{L^{\vec p}}$. Hence
\[
  a^{N_m-\lambda} a^{n_1/q_1+\ldots+n_m/q_m} \|I_{\lambda}f\|_{L^{\vec q}}
  \lesssim a^{n_1/p_1+\ldots+n_m/p_m} \|f\|_{L^{\vec p}}.
\]
Since $a$ is arbitrary, we get
(\ref{eq:homo:m}).

(S2)\, We prove that $p_i\ge 1$ for all $1\le i\le m$.

Assume on the contrary that $p_{i_0}<1$ for some $1\le i_0\le m$.
Observe that
\begin{align*}
\lambda + \sum_{\substack{1\le i\le m\\ i\ne i_0}} \frac{n_i}{p_i}
 - \sum_{\substack{1\le i\le m\\ i\ne i_0}} n_i
&= \frac{n_{i_0}}{p'_{i_0}} + \sum_{i=1}^m \frac{n_i}{q_i}
 <   \sum_{i=1}^m \frac{n_i}{q_i}.
\end{align*}
If $\vec q\ne \vec \infty$, there is some $\alpha> \sum_{\substack{1\le i\le m\\ i\ne i_0}} {n_i}/{p_i}$ such that
\begin{equation}\label{eq:ea1}
0<\lambda + \alpha
 - \sum_{\substack{1\le i\le m, i\ne i_0}} n_i
 <   \sum_{i=1}^m \frac{n_i}{q_i}.
\end{equation}
Let
\begin{equation}\label{eq:ea2}
  f(y_1,\ldots,y_m)
    = \frac{\chi^{}_{|y_{i_0}|\le 1}(y_{i_0})}
      {(1+\sum_{\substack{1\le i\le m, i\ne i_0}} |y_i|)^{\alpha}}.
\end{equation}
Then we have $f\in L^{\vec p}$.
It follows  that for any $(x_1,\ldots,x_m)\in\bbR^{N_m}$,
\begin{align}
 I_{\lambda}f(x_1,\ldots,x_m)%\nonumber  \\
&= \int_{\substack{y_i\in\bbR^{n_i}, i\ne i_0\\
   |y_{i_0}|\le 1}}
  \frac{dy_1\ldots   dy_m}
  {
  (\sum_{i=1}^m |x_i-y_i|)^{\lambda}
  (1+\sum_{\substack{1\le i\le m\\ i\ne i_0}} |y_i|)^{\alpha}}\nonumber  \\
&\gtrsim \int_{ y_i\in\bbR^{n_i}, i\ne i_0 }
  \frac{dy_1\ldots dy_{i_0-1} dy_{i_0+1}\ldots dy_m}
  {
  \big(1+|x_{i_0}|+\sum_{\substack{1\le i\le m\\ i\ne i_0}} |x_i-y_i|\big)^{\lambda}
  \big(1+\sum_{\substack{1\le i\le m\\ i\ne i_0}} |y_i|\big)^{\alpha}} \nonumber \\
&\gtrsim \int_{ y_i\in\bbR^{n_i}, i\ne i_0 }
  \frac{dy_1\ldots dy_{i_0-1} dy_{i_0+1}\ldots dy_m}
  {
  \big(1+|x_{i_0}|+\sum_{\substack{1\le i\le m\\ i\ne i_0}} (|x_i-y_i| + |y_i|)\big)^{\lambda +\alpha}} \nonumber \\
&\approx \int_{ y_i\in\bbR^{n_i}, i\ne i_0 }
  \frac{dy_1\ldots dy_{i_0-1} dy_{i_0+1}\ldots dy_m}
  {
  \big(1+|x_{i_0}|+\sum_{\substack{1\le i\le m\\ i\ne i_0}} (|x_i| + |y_i|)\big)^{\lambda +\alpha}} \label{eq:ea:a19}\\
&\approx
  \frac{1}
  {
  \big(1 +\sum_{i=1}^m |x_i|  \big)^{\lambda +\alpha-n_1-\ldots - n_{i_0-1}-n_{i_0+1}-\ldots-n_m}}.\nonumber
\end{align}
By (\ref{eq:ea1}),  $I_{\lambda} f \not\in L^{\vec q}$,
which contradicts the hypothesis.

If  $\vec q=\vec\infty$,
 then  there is some $\alpha> \sum_{ {1\le i\le m, i\ne i_0}} {n_i}/{p_i}$ such that
\[
  \lambda + \alpha
 - \sum_{ {1\le i\le m, i\ne i_0}} n_i
 < 0.
\]
Let $f$ be defined as in (\ref{eq:ea2}).
For any $(x_1,\ldots,x_m)\in\bbR^{N_m}$,
 we see from (\ref{eq:ea:a19}) that
$
  I_{\lambda} f(x_1,\ldots,x_m)
   = \infty$.
Again, we get a  contradiction.

(S3)\, We prove that
$q_i\ge 1$, $1\le i\le m$.

Assume that there is some $i_0$ such that $q_{i_0}<1$.
Set
\[
  f(y_1,\ldots,y_m) = \frac{\chi^{}_{\{  |y_i|\le 1, 1\le i\le m\}}( y_1,\ldots,y_m ) }
  {(\sum_{i\ne i_0} |y_i|)^{\alpha} },
\]
where $0<\alpha< \sum_{i\ne i_0}  n_i/p_i$.
Then $f\in L^{\vec p}$.

Suppose that $|x_i| \le 1/m$ for all $1\le i\le m$.
We have
\begin{align}
 I_{\lambda}f(x_1,\ldots,x_m)
&\ge
  \int_{\substack{|x_i-y_i|\le \sum_{j\ne i_0}|x_j|  \\
      1\le i\le m}}
  \frac{ d y_1\ldots dy_m  }
        {
        (\sum_{i\ne i_0} |y_i| )^{\alpha}
        (\sum_{i=1}^m |x_i-y_i|)^{\lambda}} \nonumber \\
&\gtrsim
  \frac{1}{(\sum_{i\ne i_0} |x_i|)^{\lambda+\alpha - N_m}}.\nonumber
 \end{align}
Note that
\begin{align*}
\lambda -N_m
   + \sum_{i\ne i_0}  \frac{n_i}{p_i}
&=  \Big(\frac{n_{i_0}}{q_{i_0}} -n_{i_0}\Big)
   + \frac{n_{i_0}}{p'_{i_0}}
   + \sum_{i\ne i_0}    \frac{n_i}{q_i}
   > \sum_{i\ne i_0}    \frac{n_i}{q_i}.
\end{align*}
There is some $0<\alpha<\sum_{i\ne i_0} n_i/p_i$
such that
$
  \lambda -N_m
   +  \alpha
   > \sum_{i\ne i_0}   {n_i}/{q_i}
$.
Consequently,
\[
\|I_{\lambda}f\|_{L^{\vec q}}
\ge
\bigg\|\frac{\chi^{}_{\{|x_i|\le  1/ m   :\, 1\le i\le m\}(x_1,\ldots,x_m)}
}{(\sum_{i\ne i_0}  |x_i|)^{\lambda-N_m+\alpha}}
\bigg \|_{L^{\vec q}}
=\infty,
\]
which contradicts the hypothesis.

(S4)\, We prove that
$q_i\ge p_i$ for all $1\le i\le m$.

First, we assume that $p_1,q_1<\infty$.
Take some $f\in L^{\vec p}$ and $z\in\bbR^{n_1}$.
Set $f_z(y_1,\ldots,y_m) = f(y_1-z,y_2,\ldots, y_m)$.
Then we have
\[
 \| I_{\lambda}f_z + I_{\lambda}f\|_{L^{\vec q}}
 \le  \|I_{\lambda}\|\cdot \|f_z + f\|_{L^{\vec p}}.
\]
By letting $|z|\rightarrow \infty$, we see from
Lemma~\ref{Lm: translation} that
\[
\lim_{|z|\rightarrow\infty} \| I_{\lambda}f_z + I_{\lambda}f\|_{L^{\vec q}}
 \le  2^{1/p_1} \|I_{\lambda}\|\cdot \|f\|_{L^{\vec p}}.
\]

On the other hand, we see from the definition of
$I_{\lambda}$ that
\begin{align*}
I_{\lambda}f_z(x_1,\ldots,x_m)
&=I_{\lambda}f(x_1-z,x_2,\ldots,x_m).
\end{align*}
Applying Lemma~\ref{Lm: translation} again, we get
\[
\lim_{|z|\rightarrow\infty}
\|I_{\lambda}f_z  +
I_{\lambda}f \|_{L^{\vec q}}
=
2^{1/q_1} \|I_{\lambda}f \|_{L^{\vec q}}.
\]
Hence
\[
  2^{1/q_1} \|I_{\lambda}f \|_{L^{\vec q}}\le 2^{1/p_1} \|I_{\lambda}\| \cdot \|f\|_{L^{\vec p}}.
\]
Therefore, $q_1\ge p_1$.

Next we assume that $p_1=\infty$.
We conclude that $q_1=\infty$ in this case.

In fact, if $1<q_1<\infty$, then $1=p'_1<q'_1<\infty$. On the other hand, we see from the duality that $I_{\lambda}$ is bounded from
$L^{\vec q\ \!\!'}$ to $L^{\vec p\ \!\!'}$. Therefore,
$p'_1\ge q'_1$, which is a contradiction.

If $q_1=1$, then $q'_1=\infty$.
For any $f(y_1,\ldots,y_m)=f_1(y_2,\ldots,y_m)\in L^{\vec p}$
and $g(x_1,\ldots,x_m)$ $=g_1(x_2,\ldots,x_m)\in L^{\vec q\ \!\!'}$, we have
\[
  \int_{\bbR^{2N_m}}
   \frac{f_1(y_2,\ldots,y_m) g_1(x_2,\ldots,x_m)}
   {(\sum_{i=1}^m|x_i-y_i|)^{\lambda}} dx_1dy_1\ldots dx_m dy_m \lesssim
   \|f\|_{L^{\vec p}} \|g\|_{L^{\vec q\ \!\!'}}.
\]
But the integration with respect to $x_1$ and $y_1$
is the infinity whenever $f,g>0$, which is a contradiction.

Similarly we can prove that $q_i\ge p_i$ for all $2\le i\le m$.

(S5)
We prove that there exist some $i$ and $j$ such that
$1<p_i<q_i $ and $p_j<q_j < \infty$.

If $p_i=q_i$ for all $1\le i\le m$, then we have $\lambda=N_m$, which is impossible since $I_{\lambda}$ is unbounded in this case.

If $\vec p =\vec 1$, then we have
$\lambda  = n_1/q_1+\ldots+n_m/q_m$.
Set $f = \chi^{}_{\{|y_i|<1:\, 1\le i\le m\}}$.
It is easy to see that $I_{\lambda}f(x) \gtrsim 1/(|x_1|+\ldots+|x_m|)^{\lambda}$ whenever
$|x_i| >1$, $1\le i\le m$. Hence
$I_{\lambda} f\not\in L^{\vec q}$, which is a contradiction. By the duality, we also have $\vec q\ne \vec\infty$.

Denote $ \{1,\ldots,m\}\setminus\{i:\, p_i=q_i, 1\le i\le m\}$ by $\{i_l:\, 1\le l\le r\}$, where $i_1<\ldots<i_r$.
By Lemma~\ref{Lm:structure}, the mapping
\[
  f\mapsto \int_{\bbR^{n_{i_1}+\ldots + n_{i_r}}}
    \frac{f(y_{i_1},\ldots,y_{i_r}) dy_{i_1}\ldots dy_{i_r}}
    {(\sum_{l=1}^r |x_{i_l}-y_{i_l}|)^{\lambda}}
\]
is bounded from $L^{(p_{i_1},\ldots,p_{i_r})}$
to  $L^{(q_{i_1},\ldots,q_{i_r})}$. Now we see from the above arguments that there exist some $i$ and $j$ such that $1<p_i<q_i$ and $p_j<q_j<\infty$.

(S6) We prove by induction on $m$ that
$(\vec p, \vec q)\in \Gamma_{\lambda,m}$.

  For $m=1$, we get the conclusion by Proposition~\ref{prop:fractional}.

Assume that the conclusion is true when $m$ is replaced by $m-1$ for some $m\ge 2$.
There are five cases.

(A1)\,\, $1<p_m<q_m<\infty$.

In this case, nothing is to be proved.

(A2)\,\, $1=p_m$ and $q_m<\infty$.

We see from (S5) that there is some $i_0$ such that $1<p_{i_0}<q_{i_0}$.
Let $i_1 = \max\{i:\, 1<p_i<q_i\}$. Then
$1\le i_1 \le m-1$
and $p_i=1$ or $p_i=q_i$ for all $i_1+1\le i\le m-1$.
Let $\Lambda_0 = \{i:\, i_1+1 \le i\le m-1\mbox{ and } p_i = q_i>1
 \}$ and $\Lambda_1 = \{1,\ldots,m\}\setminus \Lambda_0$.

For any $i\in \Lambda_0$, denote
$\{i\}\cup \Lambda_1$ by $\{j_l:\, 1\le l\le r\}$,
where $j_1<\ldots < j_r$, $r=\#\Lambda_1+1$ and $\#\Lambda_1$
stands for the cardinality of $\Lambda_1$. Applying Lemma~\ref{Lm:structure} many times, we see that
the mapping
\[
T_r:\,
  f\mapsto \int_{\bbR^{n_{j_1}+\ldots + n_{j_r}}}
    \frac{f(y_{j_1},\ldots,y_{j_r}) dy_{j_1}\ldots dy_{j_r}}
    {(\sum_{l=1}^r |x_{j_l}-y_{j_l}|)^{
     \sum_{l=1}^r n_{j_l}/p'_{j_l} + n_{j_l}/q_{j_l}
    }}
\]
is bounded from $L^{(p_{j_1},\ldots,p_{j_r})}$
to   $L^{(q_{j_1},\ldots,q_{j_r})}$.

Suppose that $i = j_{l_0}$. Then $l_0<r$
and $p_{j_l}=1$ when $l_0+1 \le l\le r$.
By Lemma~\ref{Lm:pi}, $q_{j_r}\ge p_i$. That is,
$q_m\ge p_i$ for all $i\in \Lambda_0$.

On the other hand, if we denote $ \Lambda_1$ by $\{j_l:\, 1\le l\le r\}$, where $r=\#\Lambda_1$, then $T_r$ is also bounded from $L^{(p_{j_1},\ldots,p_{j_r})}$
to   $L^{(q_{j_1},\ldots,q_{j_r})}$.
Observe that $j_l=l$ for all $1\le l\le i_1$
and $p_{j_l}=1$ for all $i_1+1\le l\le r$.
Applying Lemma~\ref{Lm:pi} again,
we get $q_m\ge p_{i_1}$.
Hence $q_m\ge p_i$ for all $i_1\le i\le m$. Therefore, $(\vec p, \vec q)$ meets (T2).

(A3)\,\, $1<p_m$ and $q_m=\infty$.

Since $I_{\lambda}$ is bounded from $L^{\vec p}$ to $L^{\vec q}$
if and only if it is bounded from $L^{\vec q\ \!\!'}$ to $L^{\vec p\ \!\!'}$,
we see from Case (A2) that
$(\vec q\ \!', \vec p\ \!')$ meets (T2). That is, $(\vec p, \vec q)$ meets (T3).

(A4)\,\, $p_m=1$ and $q_m=\infty$.

By Lemma~\ref{Lm:1:infty},
the mapping
\[
  f\mapsto \int_{\bbR^{N_m-n_m}}
    \frac{f(y_1,\ldots,y_{m-1}) dy_1\ldots dy_{m-1}}
    {(\sum_{i=1}^{m-1} |x_i-y_i|)^{\lambda}}
\]
is bounded from
$L^{(p_1,\ldots,p_{m-1})}$
to $L^{(q_1,\ldots,q_{m-1})}$.
By the inductive assumption, $(p_1$, $\ldots$, $p_{m-1}$, $q_1$, $\ldots$, $q_{m-1})\in
 \Gamma_{\lambda,m-1}$.

(A5)\,\, $p_m=q_m$.

If $p_m=1$, then we see from (A2) that $q_m>1$, which contradicts the assumption.
If $q_m=\infty$, then we see from (A3) that $p_m<\infty$, which also contradicts  the assumption.
Hence $1<p_m=q_m<\infty$.
Now we see from Lemma~\ref{Lm:structure}
that $(p_1,\ldots,p_{m-1},q_1,\ldots,q_{m-1})\in \Gamma_{\lambda-n_m,m-1}$.

\subsection{Proof of Theorem~\ref{thm:mixed fractional}:
The  Sufficiency}

Assume that $(\vec p, \vec q)\in\Gamma_{\lambda,m}$.
Denote the index set $\{i:\, 1\le i\le m, (p_i,q_i)\ne (1,\infty)\}$ by $\{j_1,\ldots,j_r\}$, where $j_1<\ldots<j_r$
and $1\le r\le m$.
It is easy to check  from the definition of $\Gamma_{\lambda,m}$ that
$(\vec p, \vec q)\in\Gamma_{\lambda,m}$ if and only if
$(p_{j_1},\ldots,p_{j_r},q_{j_1},\ldots,q_{j_r})\in\Gamma_{\lambda,r}$.
By Lemma~\ref{Lm:1:infty}, we only need to consider the case
$(p_i,q_i)\ne (1,\infty)$ for all $1\le i\le m$.
There are four cases.

(B1)\,\, $1<p_m<q_m<\infty$

For any $f\in L^{\vec p}$, since $p_1\le q_1$, we deduce from Minkowski's and Young's inequalities that
\begin{align*}
 \left\|\int_{\bbR^{N_m}}\!\!\!\!\!\! \frac{f(y_1,\ldots, y_m) dy_1\ldots dy_m}
{(\sum_{i=1}^m |x_i-y_i|)^{\lambda}} \right\|_{L_{x_1}^{q_1}}
&\le\int_{\bbR^{N_m-n_1}} \!\!\left\| \int_{\bbR^{n_1}}\!\!\frac{f(y_1,\ldots, y_m) dy_1 }
{(\sum_{i=1}^m |x_i-y_i|)^{\lambda}} \right\|_{L_{x_1}^{q_1}}\!\!dy_2\ldots dy_m \\
&\lesssim
 \int_{\bbR^{N_m-n_1}} \frac{\|f(\cdot,y_2,\ldots, y_m)\|_{L_{y_1}^{p_1}} dy_2\ldots dy_m}
{(\sum_{i=2}^{m} |x_i-y_i|)^{\lambda-n_1/p'_1-n_1/q_1}}.
\end{align*}
Similar arguments show that
\begin{align*}
\hbox to 6em{}&\hskip -6em \left\|\int_{\bbR^{N_m }}\!\! \frac{f(y_1,\ldots, y_m) dy_1\ldots dy_m}
{(\sum_{i=1}^m |x_i-y_i|)^{\lambda}} \right\|_{L_{(x_1,\ldots,x_{m-1})}^{(q_1,\ldots,q_{m-1})}}
 \lesssim
 \int_{\bbR^{n_m}}\!\! \frac{\|f(,\ldots, y_m)\|_{L_{(y_1,\ldots,y_{m-1})}^{(p_1,\ldots,p_{m-1})}}  dy_m}
{ |x_m-y_m|^{ n_m/p'_m+n_m/q_m}}.
\end{align*}
Now the conclusion follows from Proposition~\ref{prop:fractional}.

(B2)\,\, $p_m=1$ and $q_m<\infty$.

In this case, $(\vec p, \vec q)$ meets (T2).
Consequently, there is some $1\le i_1\le m-1$ such that
  $1< p_{i_1}<q_{i_1}$,
  $p_i=1$ or  $p_i=q_i $   for all
   $i_1+1\le i\le m-1$, and  $p_i\le q_m$ for all
   $i_1 \le i\le m $.

Denote $\vec{\tilde p} = (p_1,\ldots, p_{m-1})$.
By Lemma~\ref{Lm:end points}, it suffices to show that for any $f\in L^{\vec {\tilde p}}$,
\[
  \bigg\|\int_{\bbR^{N_m-n_m}}
   \frac{f(y_1,\ldots, y_{m-1}) dy_1\ldots dy_{m-1}}
{(\sum_{i=1}^{m-1} |x_i-y_i|+ |x_m|)^{\lambda}} \bigg\|_{L_{(x_1,\ldots,x_m)}^{(q_1,\ldots,q_m)}}
\lesssim \|f\|_{L^{\vec {\tilde p}}}.
\]

For any $f\in L^{\vec {\tilde p}}$, since $p_1\le q_1$, we deduce from Minkowski's and Young's inequalities that
\begin{align*}
&\bigg\|\int_{\bbR^{N_m-n_m}}\!\! \frac{f(y_1,\ldots, y_{m-1}) dy_1\ldots dy_{m-1}}
{(\sum_{i=1}^{m-1} |x_i-y_i|+ |x_m|)^{\lambda}} \bigg\|_{L_{x_1}^{q_1}} \\
&\le\int_{\bbR^{N_m-n_1-n_m}} \bigg\| \int_{\bbR^{n_1}}\frac{f(y_1,\ldots, y_{m-1}) dy_1 }
{(\sum_{i=1}^{m-1} |x_i-y_i|+ |x_m|)^{\lambda}} \bigg\|_{L_{x_1}^{q_1}}dy_2\ldots dy_{m-1} \\
&\lesssim
 \int_{\bbR^{N_m-n_1-n_m}} \frac{\|f(\cdot,y_2,\ldots, y_{m-1})\|_{L_{y_1}^{p_1}} dy_2\ldots dy_{m-1}}
{(\sum_{i=2}^{m-1} |x_i-y_i|+ |x_m|)^{\lambda-n_1/p'_1-n_1/q_1}}.
\end{align*}
Similar arguments show that
\begin{align}
\hbox to 3em{}&\hskip -3em \left\|\int_{\bbR^{N_m-n_m}}\!\! \frac{f(y_1,\ldots, y_{m-1}) dy_1\ldots dy_{m-1}}
{(\sum_{i=1}^{m-1} |x_i-y_i|+ |x_m|)^{\lambda}} \right\|_{L_{(x_1,\ldots,x_{i_1-1})}^{(q_1,\ldots,q_{i_1-1})}} \nonumber  \\
&\lesssim
 \int_{\bbR^{n_{i_1}+\ldots+n_{m-1}}}  \frac{\|f(\ldots,y_{i_1},\ldots, y_{m-1})\|_{L_{(y_1,\ldots,y_{i_1-1})}^{(p_1,\ldots,p_{i_1-1})}}  dy_{i_1}\ldots dy_{m-1}}
{( \sum_{i=i_1}^{m-1}  |x_i - y_i|+ |x_m|)^{\lambda - \sum_{i=1}^{i_1-1} (n_i/p'_i + n_i/q_i)}}. \label{eq:s:e6}
\end{align}

Define the operator $S$ by
\begin{equation}\label{eq:S:def}
  Sg(x_{i_1},\ldots,x_m) =   \int_{\bbR^{n_{i_1}+\ldots+n_{m-1}}}  \frac{ g(y_{i_1},\ldots,y_{m-1})  dy_{i_1} \ldots dy_{m-1} }
{(  \sum_{i=i_1}^{m-1} |x_i-y_i|+|x_m|)^{\alpha}},
\end{equation}
where
$
  \alpha = \lambda - \sum_{i=1}^{i_1-1}  ( {n_i}/{p'_i} + {n_i}/{q_i} )$.
%By (\ref{eq:s:e6}), we get
%\begin{equation}\label{eq:s:e6a}
%  \|I_{\lambda }f\|_{L^{\vec q}} \le
%  C_{\lambda,\vec p, \vec q}
%  \|S\|\cdot \|f\|_{L^{\vec p}}.
%\end{equation}
Let us prove that $S$ is bounded
from $L^{(p_{i_1},\ldots,p_{m-1})}$
to $L^{(q_{i_1},\ldots, q_m)}$.

For any $s>1$, define $t(s) =  n_m/(n_{i_1}/p'_{i_1} - n_{i_1}/s' + n_m/q_m)$.
Since $t(p_{i_1}) =q_m>1$,
there exist two numbers $s_0$ and $s_1$ such that
$1<s_0 < p_{i_1}< s_1 < q_{i_1}$ and   $t(s_l)\ge 1$ for $l=0,1$.

Let $\vec u_s = (s, p_{i_1+1}, \ldots, p_{m-1})$ and
$\vec v = (q_{i_1},\ldots,q_{m-1})$.
Applying Young's inequality many times, we get
\begin{equation}\label{eq:S}
\|Sg\|_{L_{(x_{i_1},\ldots,x_{m-1})}^{(q_{i_1},\ldots,q_{m-1})}}
\lesssim  \frac{ \|g\|_{L^{\vec u_{s_l}}}}{|x_m|^{n_m/t(s_l)}}.
\end{equation}
Hence $S$ is bounded from $L^{\vec u_{s_l}}$
to $L^{t(s_l),\infty}(L^{\vec v})$, $l=0,1$.

Since $s_0<p_{i_1}<s_1$, there is some $0<\theta<1$
such that $1/p_{i_1} = (1-\theta)/s_0 + \theta/s_1$.
On the other hand, it is easy to check that
$1/q_m = (1-\theta)/t(s_0) + \theta/t(s_1)$.
By the interpolation theorem (Proposition~\ref{prop:interpolation}),
$S$ is bounded from
$(L^{\vec u_{s_0}}, L^{\vec u_{s_1}})_{\theta,q_m}$
to
$(L^{t(s_0),\infty}(L^{\vec v}), L^{t(s_1),\infty}(L^{\vec v}))_{\theta,q_m}$.

Note that $p_i\le q_m$ for all $i_1\le i\le m-1$.
Applying Theorem~\ref{thm:mixed Lp}, we get
\[
   (L^{\vec u_{s_0}}, L^{\vec u_{s_1}})_{\theta,q_m}
 \hookleftarrow L^{(p_{i_1+1},\ldots,p_{m-1})}((L^{s_0},L^{s_1})_{\theta,q_m})
 \hookleftarrow L^{(p_{i_1},\ldots,p_{m-1})}.
\]
On the other hand, we see from
Corollary~\ref{Co:Lorentz} that
\[
  (L^{t(s_0),\infty}(L^{\vec v}), L^{t(s_1),\infty}(L^{\vec v}))_{\theta,q_m}
  = L^{q_m}(L^{\vec v}).
\]
Hence $S$ is bounded from $L^{(p_{i_1},\ldots,p_{m-1})}$
to $L^{q_m}(L^{\vec v})$.
It follows from (\ref{eq:s:e6}) that
\[
    \bigg\|\int_{\bbR^{N_m-n_m}}
   \frac{f(y_1,\ldots, y_{m-1}) dy_1\ldots dy_{m-1}}
{(\sum_{i=1}^{m-1} |x_i-y_i|+ |x_m|)^{\lambda}} \bigg\|_{L_{(x_1,\ldots,x_m)}^{(q_1,\ldots,q_m)}}
\lesssim   \|f\|_{L^{\vec {\tilde p}}}.
\]
Now  we see from Lemma~\ref{Lm:end points} that
$I_{\lambda}$ is bounded from $L^{\vec p}$
to $L^{\vec q}$.

(B3).\,\, $p_m>1$ and $q_m =\infty$.

In this case, $(\vec p, \vec q)$ meets (T3), which is equivalent to
$(\vec q\ \!', \vec p\ \!')$ meeting (T2). Now the conclusion follows from Case (B2) and
the duality.

(B4) $1<p_m=q_m<\infty$.

Let $m_1 = \min\{i:\, p_l=q_l,  i\le l\le m\}$.
We see from the definition of $\Gamma_{\lambda,m}$
that  $1<m_1\le m$, $1<p_i=q_i<\infty$ for all $m_1\le i\le m$
and $(p_1,\ldots, p_{m_1-1}$, $q_1$, $\ldots$, $q_{m_1-1})\in \Gamma_{\lambda - n_m - \ldots - n_{m_1}, m_1-1}$.

Note that $1\le p_{m_1-1}<q_{m_1-1}\le \infty$
and $(p_{m_1-1},q_{m_1-1})\ne (1,\infty)$.
There are three subcases.

(B4)(a)\,\, $1< p_{m_1-1}<q_{m_1-1}< \infty$.

For any $f\in L^{\vec p}$, since $p_i\le q_i$,
using Minkowski's and Young's inequalities many times, we get
that
\begin{align*}
&\left\|\int_{\bbR^{N_m}}\!\! \frac{f(y_1,\ldots, y_m) dy_1\ldots dy_m}
{(\sum_{i=1}^m |x_i-y_i|)^{\lambda}} \right\|_{L_{(x_1,\ldots,x_{m_1-2})}^{(q_1,\ldots,q_{m_1-2})}} \\
&\lesssim
 \int_{\bbR^{n_{m_1-1}+\ldots + n_m}} \frac{\|f(\ldots,y_{m_1-1},\ldots, y_m)\|_{L_{(y_1,\ldots,y_{m_1-2})}^{(p_1,\ldots,p_{m_1-2})}} dy_{m_1-1}\ldots dy_m}
{(\sum_{i=m_1-1}^{m} |x_i-y_i|
  )^{ n_{m_1-1}/p'_{m_1-1}+n_{m_1-1}/q_{m_1-1}
   + n_{m_1} + \ldots + n_{m} }}\\
&\lesssim
 \int_{\bbR^{n_{m_1-1}}} \frac{
   M_{y_{m_1}}\ldots M_{y_{m}}
   \|f(,\ldots,y_{m_1-1},x_{m_1},\ldots, x_m)\|_{L_{(y_1,\ldots,y_{m_1-2})}^{(p_1,\ldots,p_{m_1-2})}} dy_{m_1-1} }
{|x_{m_1-1}-y_{m_1-1}|^{ n_{m_1-1}/p'_{m_1-1}+n_{m_1-1}/q_{m_1-1} }}   .
\end{align*}
By Proposition~\ref{prop:fractional}, we have
\begin{align*}
&\bigg\|\Big\|\int_{\bbR^{N_m}}\!\! \frac{f(y_1,\ldots, y_m) dy_1\ldots dy_m}
{(\sum_{i=1}^m |x_i-y_i|)^{\lambda}} \Big\|_{L_{(x_1,\ldots,x_{m_1-2})}^{(q_1,\ldots,q_{m_1-2})}}
\bigg\|_{L_{x_{m_1-1}}^{q_{m_1-1}}} \\
&\lesssim
 \Big \|
   M_{y_{m_1}}\ldots M_{y_{m}}
   \|f(,\ldots,y_{m_1-1},x_{m_1},\ldots, x_m)\|_{L_{(y_1,\ldots,y_{m_1-2})}^{(p_1,\ldots,p_{m_1-2})}}
   \Big\|_{L_{y_{m_1-1}}^{p_{m_1-1}}}.
\end{align*}
Recall that $p_l=q_l$ for all $m_1\le l\le m$. It follows from
(\ref{eq:Fefferman-Stein}) that
\[
  \|I_{\lambda} f\|_{L^{\vec q}}
  \lesssim \|f\|_{L^{\vec p}}.
\]

(B4)(b)\,\, $p_{m_1-1}>1$ and $q_{m_1-1}=\infty$.

Since $(p_1,\ldots, p_{m_1-1}$, $q_1$, $\ldots$, $q_{m_1-1})\in \Gamma_{\lambda - n_m - \ldots - n_{m_1}, m_1-1}$,
there is some $1\le v \le m_1-2 $ such that
$p_v <q_v<\infty$,
$q_i=p_i $ or $q_i=\infty$ for all $v+1 \le i\le m_1-2$, and $p_{m_1-1}\le q_i$ for all $v\le i\le m_1-1$.

It suffices to consider the case $n_{m_1-1} = n_v$.
To see this, assume that $I_{\lambda}$ is bounded when $n_{m_1-1}=n_v$.
For the case  $n_{m_1-1}<n_v$,   denote $x_v=(x_{v1}, x_{v2})$
and $y_v=(y_{v1}, y_{v2})$
with $x_{v1},y_{v1}\in\bbR^{n_v-n_{m_1-1}}$
and $x_{v2}, y_{v2} \in\bbR^{n_{m_1-1}}$.
Set
\begin{align*}
\vec s &= (p_1,\ldots,p_{v-1},p_v,p_v,p_{v+1},\ldots,p_m), \\
\vec t &= (q_1,\ldots,q_{v-1},q_v,q_v,q_{v+1},\ldots,q_m).
\end{align*}
Then $(\vec s, \vec t)\in \Gamma_{\lambda,m+1}$
and
\[
    L^{\vec p}(\bbR^{n_1}\!\times\! \ldots \!\times\! \bbR^{n_m})
  = L^{\vec s}
  (\bbR^{n_1}\!\times\! \ldots\!\times\!  \bbR^{n_{v-1}}
  \!\times\! \bbR^{n_v - n_{m_1-1}}
  \!\times\! \bbR^{n_{m_1-1}}
  \!\times\!  \bbR^{n_{v+1}}\ldots  \!\times\! \bbR^{n_m}).
\]
Note that
$s_i = t_i$ for all $m_1+1\le i\le m+1$.
We have
$(s_1,\ldots, s_{m_1}$, $t_1$, $\ldots$, $t_{m_1})\in \Gamma_{\lambda - n_m - \ldots -
n_{ m_1}, m_1}$.

On the other hand, since $s_{i+1} = p_i$
and $t_{i+1}= q_i$ for all $v\le i\le m$,
by setting $\bar v=v+1$ and $\bar m_1 = m_1+1$,
we have
$s_{\bar v} <t_{\bar v}<\infty$,
$t_i=s_i $ or $t_i=\infty$ for all $\bar v+1 \le i\le \bar m_1-2$,
and $s_{\bar m_1-1}\le t_i$ for all $\bar v  \le i\le \bar m_1-1$.
Moreover, for any $ (y_1,\ldots,y_{v-1},y_{v1},y_{v2},y_{v+1},\ldots,y_m)\in \bbR^{N_m}$, the dimensions of the $\bar v$-th and the $(\bar m_1-1)$-th
components are equal.
By the assumption, $I_{\lambda}$ is bounded from $L^{\vec s}$ to $L^{\vec t}$. That is, $I_{\lambda}$ is bounded from
$L^{\vec p}$ to $L^{\vec q}$.

And for the case $n_v<n_{m_1-1}$, by denoting $x_{m_1-1}=(x_{m_1-1,1}, x_{m_1-1,2})$
with $x_{m_1-1,1}\in\bbR^{n_{m_1-1}-n_v}$
and $x_{m_1-1,2} \in\bbR^{n_v}$, similar arguments show that $I_{\lambda}$ is bounded.

Now we suppose that $n_v = n_{m_1-1}$ and
  prove the conclusion by induction on $m-m_1$. First, we
consider the case $m-m_1=0$, i.e., $m=m_1$.

Recall that $1<p_{m-1} \le q_v <\infty$ and
$n_{m-1}=n_v$.  There exist numbers
  $\tilde p_{m-1}$, $\tilde{\tilde{p}}_{m-1}$,
  $ \tilde q_v$
and  $ \tilde{\tilde q}_v$ such that
\begin{align*}
&1< \tilde p_{m-1} \le \tilde q_v <\infty,
\qquad
 1 < \tilde{\tilde{p}}_{m-1}\le \tilde{\tilde q}_v<\infty, \\
& \tilde p_{m-1}<p_{m-1}<\tilde{\tilde{p}}_{m-1} ,
\qquad  p_v< \tilde q_v < q_v < \tilde{\tilde q}_v, \\
&  \sum_{\substack{1\le i\le m \\ i\ne m -1}} \frac{n_i}{p_i}
    + \frac{n_{m-1}}{\tilde p_{m-1}}
    =  \sum_{\substack{1\le i\le m \\ i\ne v}} \frac{n_i}{q_i}
       + \frac{n_v}{\tilde q_v} + N_m-\lambda,\\
&  \sum_{\substack{1\le i\le m \\ i\ne m -1}} \frac{n_i}{p_i}
    + \frac{n_{m-1}}{\tilde{\tilde p}_{m-1}}
    =  \sum_{\substack{1\le i\le m \\ i\ne v}} \frac{n_i}{q_i}
       + \frac{n_v}{\tilde{\tilde q}_v} + N_m-\lambda.
\end{align*}
For any  $0<\theta_0\ne \theta_1<1$, set
\[
    \frac{1}{p_{\theta_i}}
    = \frac{1-\theta_i}{\tilde {p}_{m-1}}
      +\frac{\theta_i}{\tilde{\tilde {p}}_{m-1}}
   \mbox{\quad and \quad } \frac{1}{q_{\theta_i}}
    = \frac{1-\theta_i}{\tilde {q}_v}
      +\frac{\theta_i}{\tilde{\tilde {q}}_v},\quad i=0,1.
\]
We have
\[
 \sum_{\substack{1\le j\le m \\ j\ne m -1}} \frac{n_j}{p_j}
    + \frac{n_{m-1}}{p_{\theta_i}}
    =  \sum_{\substack{1\le j\le m \\ j\ne v}} \frac{n_j}{q_j}
       + \frac{n_v}{q_{\theta_i}} + N_m-\lambda,\quad i=0,1.
\]

By (\ref{eq:Fefferman-Stein}), the partial maximal operator $M_{y_m}\! $ is bounded
on $L_{y_m}^{p_m}\!(L_{y_{m-1}}^{\tilde p_{m-1}})$
and $L_{y_m}^{p_m}(L_{y_{m-1}}^{\tilde {\tilde p}_{m-1}})$,
respectively.
By  Proposition~\ref{prop:interpolation},
$M_{y_m}$ is bounded on
$(L^{p_m}(L^{\tilde{p}_{m-1}}),
L^{p_m}(L^{\tilde{\tilde{p}}_{m-1}}))_{\theta_i,1}$
for any $0<\theta_i <1$,  $i=0,1$.

On the other hand, by
 Theorem~\ref{thm:mixed Lp}
and Proposition
\ref{prop:interpolation Lorentz a},
\[
  (L^{p_m}(L^{\tilde{p}_{m-1}}),
L^{p_m}(L^{\tilde{\tilde{p}}_{m-1}}))_{\theta_i,1}
\hookrightarrow
 L^{p_m}\big((L^{\tilde{p}_{m-1}},
 L^{\tilde{\tilde{p}}_{m-1}})_{\theta_i,1}\big)
 =
 L^{p_m}(L^{p_{\theta_i},1}).
\]
Hence for any $h\in (L^{p_m}(L^{\tilde{p}_{m-1}}),
L^{p_m}(L^{\tilde{\tilde{p}}_{m-1}}))_{\theta_i,1}$, we have
\begin{equation}\label{eq:ea:a13}
    \| M_{y_m} h(y_{m-1},x_m)\|_{L_{x_m}^{p_m}(L_{y_{m-1}}^{p_{\theta_i},1})}
  \lesssim
  \|h\|_{(L^{p_m}(L^{\tilde{p}_{m-1}}),
L^{p_m}(L^{\tilde{\tilde{p}}_{m-1}}))_{\theta_i,1}}.
\end{equation}

As in Case (B4)(a),
for
  any $f\in (L^{p_m}(L^{\tilde{p}_{m-1}}),
L^{p_m}(L^{\tilde{\tilde{p}}_{m-1}}))_{\theta_i,1}(\mathcal B)
$,
where   $\mathcal B=L^{(p_1,\ldots,p_{m-2})}$ and $i$ equals $0$ or $1$,
 we have
\begin{align}
&\bigg\|\int_{\bbR^{N_m}}\!\! \frac{f(y_1,\ldots, y_m) dy_1\ldots dy_m}
{(\sum_{j=1}^m |x_j-y_j|)^{\lambda}} \bigg\|_{L_{(x_1,\ldots,x_{m -2})}^{(q_1,\ldots,q_{v-1},q_{\theta_i},q_{v+1},\ldots,q_{m -2})}} \nonumber \\
&\lesssim
 \int_{\bbR^{n_{m-1}+ n_m}} \frac{\|f(\ldots,y_{m-1},y_m)\|_{L_{(y_1,\ldots,y_{m-2})
   }^{(p_1,\ldots,p_{m-2})}} dy_{m-1}  dy_m}
{( |x_{m-1}-y_{m-1}| + |x_m-y_m|
  )^{ n_{m-1}/p'_{\theta_i}
   +   n_{m} }}\nonumber \\
&\lesssim
 \int_{\bbR^{n_{m-1}}} \frac{M_{y_m}\|f(\ldots,y_{m-1},x_m)\|_{L_{(y_1,\ldots,y_{m-2})
   }^{(p_1,\ldots,p_{m-2})}} dy_{m-1}   }
{|x_{m-1}-y_{m-1}|^{ n_{m-1}/p'_{\theta_i}
    }}\nonumber \\
&\lesssim
 \Big\|M_{y_m}\|f(\ldots,y_{m-1},x_m)\|_{L_{(y_1,\ldots,y_{m-2})
   }^{(p_1,\ldots,p_{m-2})}} \Big\|_{L_{y_{m-1}}^{p_{\theta_i},1}}
\Big\|
{|x_{m-1}-y_{m-1}|^{- \frac{n_{m-1}}{p'_{\theta_i}}
    }} \Big\|_{L_{y_{m-1}}^{p'_{\theta_i},\infty}}
   \nonumber \\
&\approx
 \Big\|M_{y_m}\|f(\ldots,y_{m-1},x_m)\|_{L_{(y_1,\ldots,y_{m-2})
   }^{(p_1,\ldots,p_{m-2})}} \Big\|_{L_{y_{m-1}}^{p_{\theta_i},1}}
    ,\qquad \forall x_{m-1}. \label{eq:ea:a9}
\end{align}
It follows from (\ref{eq:ea:a13}) that
\begin{align*}
&\Big\|\int_{\bbR^{N_m}}\!\! \frac{f(y_1,\ldots, y_m) dy_1\ldots dy_m}
{(\sum_{j=1}^m |x_j-y_j|)^{\lambda}} \Big\|_{L^{(q_1,\ldots,q_{v-1},q_{\theta_i},q_{v+1},\ldots,q_{m-2},\infty,q_m)} } \\
&\lesssim  \Big\|M_{y_m}\|f(\ldots,y_{m-1},x_m)\|_{L_{(y_1,\ldots,y_{m-2})
   }^{(p_1,\ldots,p_{m-2})}} \Big\|_{L_{x_m}^{p_m}(L_{y_{m-1}}^{p_{\theta_i},1})}\\
&\lesssim
  \|f\|_{(L^{p_m}(L^{\tilde{p}_{m-1}}),
L^{p_m}(L^{\tilde{\tilde{p}}_{m-1}}))_{\theta_i,1}(\mathcal B)
  },\quad i=0,1.
\end{align*}
By Proposition~\ref{prop:vector a},
\[
  (L^{p_m}(L^{\tilde{p}_{m-1}}),
L^{p_m}(L^{\tilde{\tilde{p}}_{m-1}}))_{\theta_i,1}
(\mathcal B)= (L^{p_m}(L^{\tilde{p}_{m-1}}(\mathcal B)),
L^{p_m}(L^{\tilde{\tilde{p}}_{m-1}}(\mathcal B)))_{\theta_i,1}.
\]
Hence
for any $f\in (L^{p_m}(L^{\tilde{p}_{m-1}}(\mathcal B)),
L^{p_m}(L^{\tilde{\tilde{p}}_{m-1}}(\mathcal B)))_{\theta_i,1}$,
\begin{align*}
&\bigg \|\int_{\bbR^{N_m}}\!\! \frac{f(y_1,\ldots, y_m) dy_1\ldots dy_m}
{(\sum_{j=1}^m |x_j-y_j|)^{\lambda}} \bigg\|_{L^{(q_1,\ldots,q_{v-1},q_{\theta_i},q_{v+1},
      \ldots,q_{m-2},\infty,q_m)} } \\
&\lesssim
  \|f\|_{(L^{p_m}(L^{\tilde{p}_{m-1}}(\mathcal B)),
L^{p_m}(L^{\tilde{\tilde{p}}_{m-1}}(\mathcal B)))_{\theta_i,1}
  },\qquad i=0,1.
\end{align*}

Choose $\theta, \theta_0, \theta_1\in (0,1)$ such that $p_{\theta_0}< p_{m-1}<p_{\theta_1}$ and
\[
  \frac{1}{p_{m-1}} = \frac{1-\theta}{p_{\theta_0}}
   + \frac{ \theta}{p_{\theta_1}}.
\]
Denote the interpolation couple $\big (L^{p_m}(L^{\tilde{p}_{m-1}}(\mathcal B)),
L^{p_m}(L^{\tilde{\tilde{p}}_{m-1}}(\mathcal B))\big)$
by $\overline{\mathcal A}$.
Applying Proposition~\ref{prop:interpolation} again, we get that $I_{\lambda}$ is bounded from
\[
  \Big(\overline{\mathcal A}_{\theta_0,1},
\overline{\mathcal A}_{\theta_1,1}
  \Big)_{\theta,p_{m-1}}
\]
to
\[
  \Big(L^{(q_1,\ldots,q_{v-1},q_{\theta_0},q_{v+1},\ldots,q_{m-2},\infty,q_m)},
  L^{(q_1,\ldots,q_{v-1},q_{\theta_1},q_{v+1},\ldots,q_{m-2},\infty,q_m)} \Big)_{\theta,p_{m-1}}.
\]

By the reiteration theorem \cite[Theorem 3.5.3]{Bergh1976}, when $p_m=p_{m-1}$,
\begin{align*}
&  \Big(\overline{\mathcal A}_{\theta_0,1},
\overline{\mathcal A}_{\theta_1,1}
  \Big)_{\theta,p_{m-1}}
 =\overline{\mathcal A}_{(1-\theta)\theta_0 + \theta\theta_1,p_{m-1}}
    =L^{\vec p},
\end{align*}
where we use Proposition~\ref{prop:vector a} and
Theorem~\ref{thm:mixed Lp} in the last step.
On the other hand,
note that $1/q_v = (1-\theta)/q_{\theta_0} + \theta/q_{\theta_1}$
and  $q_m=p_{m-1}\le q_i$ for any $v\le i \le m-1$.
Applying Proposition~\ref{prop:vector a} and
Theorem~\ref{thm:mixed Lp}  again, we get
\begin{align*}
&\Big(L^{(q_1,\ldots,q_{v-1},q_{\theta_0},q_{v+1},
\ldots,q_{m-2},\infty,q_m)},
  L^{(q_1,\ldots,q_{v-1},q_{\theta_1},q_{v+1},
  \ldots,q_{m-2},\infty,q_m)} \Big)_{\theta,p_{m-1}} \\
&=\Big(L^{(q_{\theta_0},q_{v+1},
\ldots,q_{m-2},\infty,q_m)}(L^{(q_1,\ldots,q_{v-1})}),
  L^{( q_{\theta_1},q_{v+1},
  \ldots,q_{m-2},\infty,q_m)}(L^{(q_1,\ldots,q_{v-1})}) \Big)_{\theta,p_{m-1}} \\
&
\hookrightarrow
L^{(q_v,q_{v+1},
\ldots,q_{m-2},\infty,q_m)}(L^{(q_1,\ldots,q_{v-1})})\\
&
=
 L^{\vec q}.
\end{align*}
Hence $I_{\lambda}$ is bounded from $L^{\vec p}$ to $L^{\vec q}$
when $p_m=p_{m-1}$.

Next we study the case $p_m\ne p_{m-1}$.
For any $x_m\ne y_m$ and $h\in L^{(p_1,\ldots,p_{m-1})}$,
  define the operator $T_{x_m,y_m}$ by
\begin{equation}\label{eq:vec K}
  T_{x_m,y_m} h(x_1,\ldots,x_{m-1}) =  \int_{\bbR^{N_m-n_m}}
   \frac{h(y_1,\ldots,y_{m-1}) dy_1\ldots dy_{m-1}}
   {(\sum_{i=1}^{m-1}|x_i-y_i| + |x_m-y_m|)^{\lambda}}.
\end{equation}
Since $(p_1,\ldots,p_{m-1},q_1,\ldots,q_{m-1})\in\Gamma_{\lambda-n_m,m-1}$
and
\[
  |T_{x_m,y_m} h(x_1,\ldots,x_{m-1})| \le   \int_{\bbR^{N_m-n_m}}
   \frac{|h(y_1,\ldots,y_{m-1})| dy_1\ldots dy_{m-1}}
   {(\sum_{i=1}^{m-1}|x_i-y_i| )^{\lambda-n_m}
    |x_m-y_m|^{n_m}},
\]
we see from Case (B3) that $T_{x_m,y_m}$ is bounded from
$L^{(p_1,\ldots,p_{m-1})}$
to $L^{(q_1,\ldots,q_{m-1})}$.

On the other hand, when $|x_m-y_m|\ge 2|y_m-z_m|$, we have
\begin{align*}
 &| (T_{x_m,y_m} - T_{x_m,z_m})h(x_1,\ldots,x_{m-1})| \\
& \le  \int_{\bbR^{N_m-n_m}}
   \bigg|\frac{1}{(\sum_{i=1}^{m-1} |x_i-y_i| + |x_m-y_m|)^{\lambda}}
   -  \frac{1}{(\sum_{i=1}^{m-1} |x_i-y_i| + |x_m-z_m |)^{\lambda}}
    \bigg| \\
&\qquad \times
   |h(y_1,\ldots,y_{m-1})| dy_1\ldots dy_{m-1} \\
&=\lambda   \int_{\bbR^{N_m-n_m}}
   \frac{\Big||x_m-y_m| - |x_m-z_m|\Big|\cdot |h(y_1,\ldots,y_{m-1})| dy_1\ldots dy_{m-1}}{(\sum_{i=1}^{m-1} |x_i-y_i| + \xi)^{\lambda+1}}
    \\
& \le \lambda  \int_{\bbR^{N_m-n_m}}
   \frac{|y_m-z_m| \cdot |h(y_1,\ldots,y_{m-1})| dy_1\ldots dy_{m-1}}{(\sum_{i=1}^{m-1} |x_i-y_i| + \xi)^{\lambda+1}}
    ,
\end{align*}
where we use Lagrange's mean value theorem and $\xi$ is between
$|x_m-z_m|$ and $|x_m-y_m|$. Since $|x_m-y_m|\ge 2|y_m-z_m|$, we have $ \xi \ge |x_m-y_m|/2$. Hence
\begin{align*}
 &| (T_{x_m,y_m} - T_{x_m,z_m})h(x_1,\ldots,x_{m-1})| \\
& \le \lambda 2^{\lambda+1}  \int_{\bbR^{N_m-n_m}}
   \frac{|y_m-z_m| \cdot |h(y_1,\ldots,y_{m-1})| dy_1\ldots dy_{m-1}}{(\sum_{i=1}^{m-1} |x_i-y_i|)^{\lambda-n_m} |x_m-y_m|^{n_m+1}}.
\end{align*}
It follows that
\[
  \|T_{x_m,y_m} - T_{x_m,z_m}\|_{L^{(p_1,\ldots,p_{m-1})}\rightarrow
    L^{(q_1,\ldots,q_{m-1})}}
 \lesssim  \frac{   |y_m-z_m|}{|x_m-y_m|^{n_m+1}}.
\]
Hence there is some constant $C$ such that for any
$y_m\ne z_m$,
\begin{align*}
 \int_{|x_m-y_m|\ge 2|y_m-z_m|}
 \|T_{x_m,y_m} - T_{x_m,z_m}\|_{L^{(p_1,\ldots,p_{m-1})}\rightarrow
    L^{(q_1,\ldots,q_{m-1})}}
   dx_m \le  C.
\end{align*}
Similarly we get
\begin{align*}
 \int_{|x_m-y_m|\ge 2|x_m-z_m|}
 \|T_{x_m,y_m} - T_{z_m,y_m}\|_{L^{(p_1,\ldots,p_{m-1})}\rightarrow
    L^{(q_1,\ldots,q_{m-1})}}
   dy_m \le  C .
\end{align*}

Recall that $I_{\lambda}$ is bounded from $L^{\vec p}$ to
$L^{\vec q}$ whenever $p_m = q_m = p_{m-1}$.
By Proposition~\ref{prop:singular T},
$I_{\lambda}$ is bounded from $L^{\vec p}$ to
$L^{\vec q}$ whenever $1<p_m = q_m<\infty$.

Now assume that for some $k\ge 0$, $I_{\lambda}$ is bounded from $L^{\vec p}$
to $L^{\vec q}$ when $m-m_1 = k\ge 0$,
$(\vec p, \vec q)\in\Gamma_{\lambda,m}$
satisfying $p_i=q_i$ for all $m_1\le i\le m$, $p_{m_1-1}>1$
and $q_{m_1-1}=\infty$.

Let us prove that $I_{\lambda}$ is also bounded  when $m-m_1 = k+1$.

First, we study the case $p_m=p_{m-1}$. Set $\tilde x_{m-1}
= (x_{m-1}, x_m)$ and  $\tilde y_{m-1}
= (y_{m-1}, y_m)$. We have
\begin{align*}
I_{\lambda} f(x_1, \ldots,x_m)
 & = \int_{\bbR^{N_m}} \frac{f(y_1,\ldots,y_m) dy_1\ldots dy_m}
   {(\sum_{i=1}^m |x_i-y_i|)^{\lambda}} \\
&\approx
I_{\lambda} f(x_1, \ldots,x_{m-2}, \tilde x_{m-1}) \\
 & = \int_{\bbR^{N_m}} \frac{f(y_1,\ldots,y_{m-2}, \tilde y_{m-1}) dy_1\ldots dy_{m-2} d\tilde y_{m-1}}
   {(\sum_{i=1}^{m-2} |x_i-y_i| + |\tilde x_{m-1} - \tilde y_{m-1}|)^{\lambda}} .
\end{align*}
By the inductive assumption, we get
\begin{align*}
\|I_{\lambda}f\|_{L^{\vec q}}
& \approx \|I_{\lambda} f(x_1, \ldots,x_{m-2}, \tilde x_{m-1})\|_{L_{(x_1, \ldots,x_{m-2}, \tilde x_{m-1})}^{
  (q_1, \ldots,q_{m-2}, q_{m-1})}} \\
&\lesssim  \| f(y_1, \ldots,y_{m-2}, \tilde y_{m-1})\|_{L_{(y_1, \ldots,y_{m-2}, \tilde y_{m-1})}^{
  (p_1, \ldots,p_{m-2}, p_{m-1})}} \\
&=  \|f\|_{L^{\vec p}},\qquad \forall f\in L^{\vec p}.
\end{align*}
Hence $I_{\lambda}$ is bounded from $L^{\vec p}$ to
$L^{\vec q}$ when $p_m = p_{m-1}$.

For the case $p_m\ne p_{m-1}$, define the operator $T_{x_m,y_m}$ by (\ref{eq:vec K}).
With almost verbatim proof as the case $m-m_1=0$ except that the boundedness of
$T_{x_m,y_m}$ from $L^{(p_1,\ldots,p_{m-1})}$ to $L^{(q_1,\ldots,q_{m-1})}$ is now deduced from the inductive assumption other than Case (B3), we get that
$I_{\lambda}$ is bounded from $L^{\vec p}$ to
$L^{\vec q}$.

(B4)(c)\,\, $p_{m_1-1}=1$ and $q_{m_1-1}<\infty$.

We see from Case (B4)(b) that $I_{\lambda}$
is bounded from $L^{\vec q\ \!\!'}$ to $L^{\vec p\ \!\!'}$.
Now the conclusion follows by the duality.
This completes the proof.

\subsection{Proof of Theorem~\ref{thm:HLS}}

Since $(\vec p, \vec q)\in \Omega_m$
is equivalent to
$(\vec p, \vec q\ \!' ) \in \Gamma_{\lambda,m}$
with $\lambda = \sum_{i=1}^m (n_i/p'_i+n_i/q'_i)$,
the sufficiency follows from Theorem~\ref{thm:mixed fractional}.

For the necessity, we only need to show that
$p_i, q_i\ge 1$ for all $1\le i\le m$.

Assume   $0<p_{i_0}<1$ for some $1\le i_0\le m$.
Fix some $\tilde h \in L^{(p_1,\ldots,p_{i_0-1},p_{i_0+1},\ldots,
p_m)}$ and $g\in L^{\vec q}$ which are non-negative and not identical to $0$.
Denote $\tilde y = (y_1$, $\ldots$, $y_{i_0-1}$, $y_{i_0+1},\ldots,y_m)$.
For any $h\in L^{p_{i_0}}$, we have
$  h(y_{i_0})
  \tilde h(\tilde y) \in L^{\vec p}$.
Let
\[
  \mu (h) =
\int_{\bbR^{2N_m}}\!\!\!\! \frac{h(y_{i_0})
  \tilde h(\tilde y) g(x_1,\ldots,x_m)
     dy_1\ldots dy_m dx_1\ldots dx_m}
     {(\sum_{i=1}^m|x_i-y_i|)^{\sum_{i=1}^m (n_i/p'_i+n_i/q'_i)}} .
\]
We have
\[
  |\mu(h)| \lesssim \|\tilde h\|_{L^{(p_1,\ldots,p_{i_0-1},p_{i_0+1},\ldots,
  p_m)}} \|g\|_{L^{\vec q}}\cdot \|h\|_{L^{p_{i_0}}},
  \qquad \forall h\in L^{p_{i_0}}.
\]
That is, $\mu$ is a non-zero
bounded linear functional on $L^{p_{i_0}}$,
which contradicts the fact that $(L^{p_{i_0}})^*=\{0\}$ (see
\cite[Theorem 1.4.16]{Grafakos2014}).

Since $\vec p$ and $\vec q$ are symmetric, we also have $q_i\ge 1$
for all $1\le i\le m$.
This completes the proof.

\footnotesize

% \bibliographystyle{abbrv}
% \bibliography{harmonicA}

\begin{thebibliography}{10}

\bibitem{AdamsBagby1974}
D.~R. Adams and R.~J. Bagby.
\newblock Translation-dilation invariant estimates for {R}iesz potentials.
\newblock {\em Indiana Univ. Math. J.}, 23(11):1051--1067, 1974.

\bibitem{Bagby1975}
R.~J. Bagby.
\newblock An extended inequality for the maximal function.
\newblock {\em Proc. Amer. Math. Soc.}, 48:419--422, 1975.

\bibitem{BandaliyevSerbetci2018}
R.~A. Bandaliyev, A.~Serbetci, and S.~G. Hasanov.
\newblock On {H}ardy inequality in variable {L}ebesgue spaces with mixed norm.
\newblock {\em Indian J. Pure Appl. Math.}, 49(4):765--782, 2018.

\bibitem{Benedek1962}
A.~{Benedek}, A.~P. {Calder\'on}, and R.~{Panzone}.
\newblock {Convolution operators on Banach space valued functions.}
\newblock {\em {Proc. Natl. Acad. Sci. USA}}, 48:356--365, 1962.

\bibitem{Benedek1961}
A.~Benedek and R.~Panzone.
\newblock The spaces {$L^P$}, with mixed norm.
\newblock {\em Duke Math. J.}, 28:301--309, 1961.

\bibitem{Bergh1976}
J.~Bergh and J.~L\"{o}fstr\"{o}m.
\newblock {\em Interpolation spaces. {A}n introduction}.
\newblock Springer-Verlag, Berlin-New York, 1976.
\newblock Grundlehren der Mathematischen Wissenschaften, No. 223.

\bibitem{Boggarapu2017}
P.~Boggarapu, L.~Roncal, and S.~Thangavelu.
\newblock Mixed norm estimates for the {C}es\`aro means associated with
  {D}unkl-{H}ermite expansions.
\newblock {\em Trans. Amer. Math. Soc.}, 369(10):7021--7047, 2017.

\bibitem{CarneiroOliveiraSousa2019}
E.~Carneiro, D.~Oliveira~e Silva, and M.~Sousa.
\newblock Sharp mixed norm spherical restriction.
\newblock {\em Adv. Math.}, 341:583--608, 2019.

\bibitem{ChenSun2019}
T.~Chen and W.~Sun.
\newblock Extension of multilinear fractional integral operators to linear
  operators on lebesgue spaces with mixed norms.
\newblock {\em arXiv preprint, arXiv:1902.04527}, 2019.

\bibitem{ChenSun2017}
T.~Chen and W.~Sun.
\newblock Iterated weak and weak mixed-norm spaces with applications to
  geometric inequalities.
\newblock {\em J. Geom. Anal.}, In Press, 2019.

\bibitem{Ciaurri2017}
O.~Ciaurri, A.~Nowak, and L.~Roncal.
\newblock Two-weight mixed norm estimates for a generalized spherical mean
  {R}adon transform acting on radial functions.
\newblock {\em SIAM J. Math. Anal.}, 49(6):4402--4439, 2017.

\bibitem{Cleanthous2017}
G.~Cleanthous, A.~G. Georgiadis, and M.~Nielsen.
\newblock Anisotropic mixed-norm {H}ardy spaces.
\newblock {\em J. Geom. Anal.}, 27(4):2758--2787, 2017.

\bibitem{Cordoba2017}
A.~C\'{o}rdoba and E.~Latorre~Crespo.
\newblock Radial multipliers and restriction to surfaces of the {F}ourier
  transform in mixed-norm spaces.
\newblock {\em Math. Z.}, 286(3-4):1479--1493, 2017.

\bibitem{Cwiel1974}
M.~Cwikel.
\newblock On {$(L^{p_0}(A_{0}),\,\ L^{p_{1}}(A_{1}))_{\theta , q}$}.
\newblock {\em Proc. Amer. Math. Soc.}, 44(2):286--292, 1974.

\bibitem{Duoandikoetxea2001}
J.~{Duoandikoetxea}.
\newblock {\em {Fourier analysis. Transl. from the Spanish and revised by David
  Cruz-Uribe.}}, volume~29.
\newblock Providence, RI: American Mathematical Society (AMS), 2001.

\bibitem{FeffermanStein1971}
C.~Fefferman and E.~M. Stein.
\newblock Some maximal inequalities.
\newblock {\em Amer. J. Math.}, 93:107--115, 1971.

\bibitem{Fernandez1987}
D.~L. Fernandez.
\newblock {Vector-valued singular integral operators on $L\sp p$-spaces with
  mixed norms and applications.}
\newblock {\em {Pac. J. Math.}}, 129(2):257--275, 1987.

\bibitem{Georgiadis2017}
A.~G. Georgiadis, J.~Johnsen, and M.~Nielsen.
\newblock Wavelet transforms for homogeneous mixed-norm {T}riebel-{L}izorkin
  spaces.
\newblock {\em Monatsh. Math.}, 183(4):587--624, 2017.

\bibitem{Grafakos2014}
L.~Grafakos.
\newblock {\em Classical {F}ourier analysis}, volume 249 of {\em Graduate Texts
  in Mathematics}.
\newblock Springer, New York, third edition, 2014.

\bibitem{Grafakos2014m}
L.~Grafakos.
\newblock {\em Modern {F}ourier analysis}, volume 250 of {\em Graduate Texts in
  Mathematics}.
\newblock Springer, New York, third edition, 2014.

\bibitem{Hart2018}
J.~Hart, R.~H. Torres, and X.~Wu.
\newblock Smoothing properties of bilinear operators and {L}eibniz-type rules
  in {L}ebesgue and mixed {L}ebesgue spaces.
\newblock {\em Trans. Amer. Math. Soc.}, 370(12):8581--8612, 2018.

\bibitem{Ho2018}
K.-P. Ho.
\newblock Mixed norm {L}ebesgue spaces with variable exponents and
  applications.
\newblock {\em Riv. Math. Univ. Parma (N.S.)}, 9(1):21--44, 2018.

\bibitem{Hormander1960}
L.~H\"{o}rmander.
\newblock Estimates for translation invariant operators in {$L^{p}$}\ spaces.
\newblock {\em Acta Math.}, 104:93--140, 1960.

\bibitem{HuangLiuYangYuan2018}
L.~Huang, J.~Liu, D.~Yang, and W.~Yuan.
\newblock Atomic and littlewood-paley characterizations of anisotropic
  mixed-norm hardy spaces and their applications.
\newblock {\em J. Geom. Anal.}, 29(3):1991--2067, 2019.

\bibitem{HuangLiuYangYuan2019}
L.~Huang, J.~Liu, D.~Yang, and W.~Yuan.
\newblock Dual spaces of anisotropic mixed-norm {H}ardy spaces.
\newblock {\em Proc. Amer. Math. Soc.}, 147(3):1201--1215, 2019.

\bibitem{HuangYang2019}
L.~Huang and D.~Yang.
\newblock On function spaces with mixed norms --- a survey.
\newblock {\em arXiv preprint arXiv:1908.03291}, 2019.

\bibitem{Hytonen2016}
T.~Hyt\"{o}nen, J.~van Neerven, M.~Veraar, and L.~Weis.
\newblock {\em Analysis in {B}anach spaces. {V}ol. {I}. {M}artingales and
  {L}ittlewood-{P}aley theory}, volume~63 of {\em Ergebnisse der Mathematik und
  ihrer Grenzgebiete. 3. Folge. A Series of Modern Surveys in Mathematics
  [Results in Mathematics and Related Areas. 3rd Series. A Series of Modern
  Surveys in Mathematics]}.
\newblock Springer, Cham, 2016.

\bibitem{Janson1988}
S.~Janson.
\newblock On interpolation of multilinear operators.
\newblock In {\em Function spaces and applications ({L}und, 1986)}, volume 1302
  of {\em Lecture Notes in Math.}, pages 290--302. Springer, Berlin, 1988.

\bibitem{Johnsen2015}
J.~Johnsen, S.~Munch~Hansen, and W.~Sickel.
\newblock Anisotropic {L}izorkin-{T}riebel spaces with mixed norms---traces on
  smooth boundaries.
\newblock {\em Math. Nachr.}, 288(11-12):1327--1359, 2015.

\bibitem{KarapetyantsSamko2018}
A.~N. Karapetyants and S.~G. Samko.
\newblock On mixed norm {B}ergman-{O}rlicz-{M}orrey spaces.
\newblock {\em Georgian Math. J.}, 25(2):271--282, 2018.

\bibitem{Kurtz2007}
D.~S. {Kurtz}.
\newblock {Classical operators on mixed-normed spaces with product weights.}
\newblock {\em {Rocky Mt. J. Math.}}, 37(1):269--283, 2007.

\bibitem{Lechner2018}
R.~Lechner.
\newblock Factorization in mixed norm {H}ardy and {BMO} spaces.
\newblock {\em Studia Math.}, 242(3):231--265, 2018.

\bibitem{LiStinga2017}
P.~Li, P.~R. Stinga, and J.~L. Torrea.
\newblock On weighted mixed-norm {S}obolev estimates for some basic parabolic
  equations.
\newblock {\em Commun. Pure Appl. Anal.}, 16(3):855--882, 2017.

\bibitem{Lions1964}
J.-L. Lions and J.~Peetre.
\newblock Sur une classe d'espaces d'interpolation.
\newblock {\em Inst. Hautes \'{E}tudes Sci. Publ. Math.}, 19:5--68, 1964.

\bibitem{Milman1981}
M.~Milman.
\newblock On interpolation of {$2^{n}$} {B}anach spaces and {L}orentz spaces
  with mixed norms.
\newblock {\em J. Functional Analysis}, 41(1):1--7, 1981.

\bibitem{Muscalu2013i}
C.~Muscalu and W.~Schlag.
\newblock {\em Classical and multilinear harmonic analysis. {V}ol. {I}}, volume
  137 of {\em Cambridge Studies in Advanced Mathematics}.
\newblock Cambridge University Press, Cambridge, 2013.

\bibitem{Pisier1993}
G.~Pisier.
\newblock The $k_t$-functional for the interpolation couple $l_1(a_0)$,
  $l_{\infty}(a_1)$.
\newblock {\em J. Approx. Theory}, 73(1):106--117, 1993.

\bibitem{Rubio1986}
J.~L. {Rubio de Francia}, F.~J. {Ruiz}, and J.~L. {Torrea}.
\newblock {Calder\'on-Zygmund theory for operator-valued kernels.}
\newblock {\em {Adv. Math.}}, 62:7--48, 1986.

\bibitem{Rudin1987}
W.~Rudin.
\newblock {\em Real and Complex Analysis}.
\newblock McGraw-Hill Companies, Inc., third edition, 1987.

\bibitem{Sagher1972}
Y.~Sagher.
\newblock Interpolation of $r$-banach spaces.
\newblock {\em Studia Math.}, 41:45--70, 1972.

\bibitem{Sandik2018}
A.~Sandik\c{c}i.
\newblock On the inclusions of some {L}orentz mixed normed spaces and
  {W}iener-{D}itkin sets.
\newblock {\em J. Math. Anal.}, 9(2):1--9, 2018.

\bibitem{Stefanov2004}
A.~Stefanov and R.~H. Torres.
\newblock Calder\'{o}n-{Z}ygmund operators on mixed {L}ebesgue spaces and
  applications to null forms.
\newblock {\em J. London Math. Soc. (2)}, 70(2):447--462, 2004.

\bibitem{Stein1970}
E.~M. Stein.
\newblock {\em Singular integrals and differentiability properties of
  functions}.
\newblock Princeton Mathematical Series, No. 30. Princeton University Press,
  Princeton, N.J., 1970.

\bibitem{Torres2015}
R.~H. {Torres} and E.~L. {Ward}.
\newblock {Leibniz's rule, sampling and wavelets on mixed Lebesgue spaces.}
\newblock {\em {J. Fourier Anal. Appl.}}, 21(5):1053--1076, 2015.

\bibitem{WeiYan2018}
M.~Wei and D.~Yan.
\newblock The boundedness of two classes of oscillator integral operators on
  mixed norm space.
\newblock {\em Adv. Math. (China)}, 47(1):71--80, 2018.

\end{thebibliography}
% \end{document}

\end{document}